\documentclass[11pt]{amsart}
\usepackage[margin=1in]{geometry}
\usepackage{amsmath,amsthm,amssymb}
\usepackage{hyperref}
\usepackage[T1]{fontenc}
\usepackage{booktabs}
\usepackage{graphicx}
\usepackage{xcolor}
\usepackage{tikz}
\usepackage{multicol}
\usepackage{mathtools}
\usepackage{placeins}
\usepackage{epstopdf}
\usepackage{xspace}
\usepackage{siunitx}
\usepackage{subcaption}
\usepackage{natbib}

\DeclareGraphicsExtensions{.pdf,.png,.jpg,.eps}

\makeatletter
\renewcommand\paragraph{\@startsection{paragraph}{4}{\z@}%
  {1.5ex \@plus .5ex \@minus .2ex}%
  {.7ex \@plus .2ex}%
  {\normalfont\normalsize\bfseries}}
\makeatother

\newcommand{\R}{\mathbb{R}}
\newcommand{\N}{\mathbb{N}}

\newcommand{\C}{\mathcal{C}}
\newcommand{\G}{\mathcal{G}}
\newcommand{\WG}{\widetilde{\mathcal{G}}}
\newcommand{\WO}{\widetilde{\Omega}}
\newcommand{\E}{\mathcal{E}}
\newcommand{\str}{\mathop{\rm str}}
\newcommand{\wk}{\mathop{\rm wk}}
\newcommand{\dcand}{D_B}
\newcommand{\dcandbar}{\overline{D}_B}

\newcommand{\dHaus}{d_{\hbox{\rm \scriptsize Haus}}}
\newcommand{\bb}{\mathfrak{b}}

\newcommand{\hh}{\mathfrak{h}}
\DeclareMathOperator{\avg}{avg}
\newcommand{\sigmarev}{\sigma^{\mathrm{rev}}}

\newcommand{\Con}{{\tt Con}\xspace}
\newcommand{\LD}{{\tt LD}\xspace}
\newcommand{\SNP}{{\tt SNP}\xspace}
\newcommand{\Lab}{{\tt Lab}\xspace}
\newcommand{\Grn}{{\tt Grn}\xspace}

\theoremstyle{plain}
\newtheorem{theorem}{Theorem}[section]
\newtheorem{proposition}[theorem]{Proposition}
\newtheorem{corollary}[theorem]{Corollary}
\theoremstyle{definition}
\newtheorem{definition}[theorem]{Definition}
\newtheorem{example}[theorem]{Example}
\theoremstyle{remark}
\newtheorem{remark}[theorem]{Remark}

\definecolor{alizarin}{rgb}{0.36, 0.54, 0.66}
\definecolor{Labcolor}{HTML}{E32636}
\definecolor{Concolor}{HTML}{0F4D92}
\definecolor{LDcolor}{HTML}{FF9933}
\definecolor{Gcolor}{HTML}{4CBB17}
\definecolor{SNPcolor}{HTML}{FFE135}
\definecolor{Indcolor}{HTML}{008B8B}

\title[Metric geometry for ranking-based voting]{Metric geometry for ranking-based voting:
Tools for learning electoral structure}
\author{Moon Duchin}
\address{Data Science Institute, University of Chicago}
\email{mduchin@uchicago.edu}
\author{Kristopher Tapp}
\address{Department of Mathematics, Saint Joseph's University}
\email{ktapp@sju.edu}

\begin{document}

\maketitle

\begin{abstract}
In this paper, we develop the metric geometry of ranking statistics, proving that the two major permutation distances in the statistics literature---Kendall tau and Spearman footrule---extend naturally to incomplete rankings with both {\em coordinate embeddings} and sparse {\em graph realizations}. The graph realizations are built from natural local ballot moves, and for partial ballots, only one of the standard scoring conventions is compatible with this local structure. This gives us a common geometric framework for ballots, candidates, and slates in which polarization and proportionality acquire natural interpretations.  As an important application, the metric structure enables efficient identification of {\em blocs} of voters and {\em slates} of their preferred candidates. Since the definitions work for partial ballots, we can execute the methods not only on synthetic elections, but on a suite of real-world elections. This gives us robust clustering methods that often produce an identical grouping of voters---even though one family of methods is based on a Condorcet-consistent ranking rule while the other is not.
\end{abstract}

\section{Introduction}\label{S:intro}

There is a significant literature using geometry to study ranked preferences, including spatial models in computational social choice theory and ideal-point models in political science. These approaches treat voters and candidates as occupying positions in a common, unobserved metric space, with the votes providing evidence about that latent geometry. In this paper, we develop the reverse point of view. For us, the “truth of the matter” is the rankings themselves. Rather than seeking to explain them through estimated locations in an unknown space, we draw on classical ranking statistics to develop two canonical geometries that faithfully record the full ordinal data while presenting alternatives for how to metrize their differences.

Two classical notions of difference between complete rankings are Kendall tau, which counts pairwise disagreements, and Spearman footrule, which measures positional displacement. We develop parallel coordinate and graph realizations of these metrics on a space that also includes the partial (or “top-truncated”) ballots common in real ranked-choice elections. The graph realization is particularly revealing. A sparse graph built from natural local ballot moves recovers Kendall distance as its path metric; adding weighted shortcut edges recovers Borda distance, which is identical to Spearman footrule on complete rankings, up to a scaling factor. For partial ballots, this construction distinguishes between two standard Borda conventions: the pessimistic convention admits the sparse local graph realization, while the averaged convention does not. Thus the graph perspective exposes structure that is not apparent from the coordinate formulas alone.

Given a metric-space setting with positions for votes, for candidates, and for winners under a voting rule,  many descriptive measures of preference and satisfaction acquire natural geometric interpretations.
{\em Polarization} concerns the separation between voter blocs relative to their internal dispersion; {\em diversity} concerns the coverage of the preference profile across possible ballots; and {\em proportionality} compares the representation received by candidate slates with the size and cohesion of their supporting voter blocs. As a first application of our framework, we develop methods for grouping similar voters and candidates into blocs and slates needed to work with these concepts. We test these methods on synthetic preference profiles and on more than 1000 ranked-choice local-government elections in Scotland, which feature five major parties and dozens of minor ones.
Because the methods do not make use of party labels, the labels provide an external check on our identification of slates, while the analysis can also identify structure missed by a party-based classification.

By a {\em ballot}, we mean a complete or partial ranking of $m$ candidates (with partial rankings sometimes called ``top-truncated ballots'' by other authors).  
We treat all unranked candidates as tied below the ranked candidates, so a ballot of length $m-1$ is identified with its unique completion to a ballot of length $m$. We also consider {\em generalized} ballots (sometimes called ``weak-order ballots''), which allow ties in arbitrary positions. A {\em preference profile} is a finite collection (i.e., multiset) of ballots.

We will introduce coordinate embeddings that correspond to the two most important notions of ranking distance in the statistics literature, making sure that they seamlessly extend from complete to partial ballots.

The first metric embedding assigns a ballot to a point in $\R^m$ by taking each candidate's coordinate to be the reverse of their ranking, so that a candidate earns a score of $m-j$ from a voter who ranks them $j$th. Because this system of scoring is sometimes called (standard) {\em Borda points} in the context of voting, we call this the {\em Borda embedding}. An alternative system of coordinates maps a ballot to $\{-1,0,1\}^{\binom m2}\subset \R^{\binom m2}$, where each pairwise comparison is recorded as favoring one of the two candidates ($\pm 1$) or as a tie ($0$). We call this the {\em head-to-head embedding}. We then define the distances $d_B$ and $d_H$ between ballots as half the $L^1$ distance between the corresponding vectors. These recover Spearman footrule and Kendall tau (also known as swap distance, or bubble sort distance), respectively.

As an alternative to coordinate embeddings, we realize $d_B$ and $d_H$ as graph metrics. We define a sparse graph called the {\em ballot graph} whose nodes are partial rankings and whose path metric recovers $d_H$; by adding edges, we obtain a {\em shortcut ballot graph} that recovers $d_B$.

A preference profile sits as a point cloud in the associated coordinate space.  An optimal center among complete ballots under the $L^1$ metric in the head-to-head embedding is precisely the classical {\em Kemeny ranking}, and $k$-clustering in this setting is the $k$-Kemeny problem.
In the Borda embedding, it is the least-squares ($L^2$) minimizer that would recover the usual Borda ranking rule---the coordinates record average Borda score for candidates.  $L^1$ minimizers give a Borda median ranking; this has been studied under the heading of rank aggregation but not fully embraced in social choice theory as a voting rule.\footnote{For complete rankings, this median was
proposed as a consensus method by \citet{cook-seiford1978} in an operations research setting; they showed that it
can be computed by solving an assignment problem.}  

\subsection{Contributions}

There are several main contributions.
\begin{itemize}
\item {\bf Common geometric spaces for ranked preferences.} We construct two explicit, lossless embeddings that place complete, partial, and generalized ballots into common ambient spaces. Preference profiles become weighted point clouds, aggregation rules become centroid and median constructions, and candidate slates themselves become distinguished points. We  prove sharp uniform bounds relating the positional and pairwise geometries. This brings several usually separate notions---including blocs, slates, polarization, diversity, and proportionality---into a common geometric framework.
\item {\bf Sparse graph realizations.} We give intrinsic combinatorial realizations of both ballot geometries: the head-to-head and Borda distances arise exactly as path metrics on sparse graphs of natural ballot moves. Crucially, extending the construction to partial ballots singles out the pessimistic Borda convention, which is compatible with these local moves, while the averaged convention is not. We also extend the graph realization of head-to-head distance to generalized ballots.
\item {\bf Learning electoral structure.} We use the resulting metric spaces to develop methods for identifying voter blocs and candidate slates. Synthetic elections provide controlled examples, while a dataset of more than 1000 Scottish local-government elections allows us to compare the methods at scale. The positional and pairwise geometries frequently produce identical centers and closely matching voter clusters, and the party-blind slate methods overwhelmingly keep candidates from the same party together while also revealing inter-party structure.
\end{itemize}

Sometimes extensions from complete to partial rankings are cosmetic or routine, but here the emphasis on seamless extension forces a choice and reveals structure that is invisible in the complete-permutation setting.

Overall, we develop a common geometric framework for observed rankings, apply it to the unsupervised learning problem of detecting blocs and slates, and obtain stable and interpretable results across synthetic and real-world elections.
All data and code used in this paper are available at \url{https://github.com/mggg/ballot-clustering}

\subsection{Related work} 

\paragraph{Ranking metrics}
Well-studied metrics on full permutations (i.e., complete ballots) include the {\em Spearman footrule} and the {\em Kendall tau distance} \citep{Kendall_1970,Kendall_Gibbons_1990}.
The classic paper of \citet{Diaconis_Graham_1977} calls them $D$ and $I$, respectively, and establishes that the two metrics are mutually bounded.
These agree with our head-to-head metric $d_H$ and our Borda metric $d_B$ via $d_H=I$ and $d_B=\frac 12 D$.
In our notation, Diaconis--Graham establish sharp bounds
$$\frac 12 d_H \leq d_B \leq d_H.$$
The same bounds were extended to top-$k$ rankings by \citet{Fagin_Kumar_Sivakumar_2003} and to rankings with general ties by a larger set of co-authors in \citet{Fagin_Kumar_Mahdian_Sivakumar_Vee_2006}.
Both papers score a tied or truncated ballot by the \emph{averaged} convention, awarding each candidate the mean of the Borda points it would receive over all resolutions of the ties.
We instead adopt the \emph{pessimistic} convention, which awards the minimum---in particular, zero points to every unmentioned candidate \citep{Kamwa_2022}.
The Diaconis--Graham bounds survive this change, and the choice of convention turns out to matter for the graph realizations below.

\paragraph{Geometry of the ranking space}
The two coordinate embeddings used here both have antecedents.
The Borda embedding is essentially the permutation-matrix representation of \citet{Young_1975}, who obtains a ballot's score vector by applying its permutation matrix to a positional scoring vector.
\citet{Saari_1995} develops a geometry of preference profiles that separates the information used by positional rules from that used by pairwise rules.
\citet{Zwicker_2008} places ballots and possible outcomes in a common Euclidean space in which both positional scoring rules and the Kemeny rule select the outcome nearest the mean embedded ballot.
The {\em permutahedron}, the combinatorial object whose vertices are the complete rankings and whose 1-skeleton is the neighbor-swap graph, sits at the meeting point of the coordinate and graph viewpoints; it has been used to relate the Borda count and the Kemeny rule \citep{Crisman_2014}, to organize models of choice \citep{Zhang_2004}, and to analyze ranked data \citep{Chen_etal_2021}. These works use geometric representations chiefly to analyze the behavior of voting rules. We instead treat the embedded ballot spaces as primary objects: within each geometry, ballots, candidates, slates, and profiles can be studied together, prior to any aggregation into outcomes.

\paragraph{Graphs of rankings}
Graphs on rankings also underlie the study of restricted preference domains: \citet{Galambos_Reiner_2008} use the weak Bruhat order in their study of acyclic domains, while \citet{Puppe_2018} uses connectivity under neighbor swaps to characterize the single-peaked domain.
Beyond linear orders, \citet{Kemeny_1959} defined a pairwise-disagreement distance on weak orders and illustrated it with a graph in the case of $m=3$ candidates, though without a general construction.
Graphs of weak orders joined by elementary moves appear in the median-procedure literature \citep{Barthelemy_Monjardet_1981} and in mathematical psychology \citep{Doignon_Falmagne_2011}.
These works use graphs to encode adjacency, order, or domain structure among rankings.
Here, we ask which ranking distances are recovered exactly as path metrics on sparse graphs of natural ballot moves, including truncation.
Answering this places Kendall and Spearman in a common graph framework and, for partial ballots, singles out the pessimistic Borda convention.

\paragraph{Consensus and centers}
Taking a ranking that minimizes summed distance to the input rankings as a consensus is an old idea; \citet{Fishburn_1977} examines the axiomatic properties of minimization methods as voting rules, and with Kendall distance the method is Kemeny's rule.
The same problem was later taken up in computing under the name of {\em rank aggregation} \citep{Dwork2001RankAggregation,Van_Zuylen_Williamson_2009}.
The distance-rationalization framework surveyed by \citet{Elkind_Slinko_2016} generalizes this: a distance on individual rankings is lifted to a distance between preference profiles, and a voting rule selects the outcome associated with the nearest profile in a specified consensus class, so that both Kemeny and Borda arise from ballot metrics.
We use the same metrics but do not pass through a rule.
The center of a cluster is a summary of the voting behavior within that cluster, whether or not it is a ranking and whether or not the voters are closely concentrated nearby.

\paragraph{Spatial models and distortion}
A distinct geometric tradition begins with spatial models of preference, as described above: voters and candidates occupy an unobserved metric space, and the observed rankings provide partial information about that latent geometry.
The literature on metric distortion asks how well a voting rule performs across all metric realizations compatible with a given preference profile; see \citet{anshelevich2021distortion} for a survey.
Our metrics live on the rankings themselves, with no latent space.

\paragraph{Polarization and clustering}
The identification of blocs and slates provides grounding for measuring proportionality, diversity, and polarization, which are active topics in computational social choice \citep{hashemi2014measuring,karpov2017preference,skowron2017proportional,israel2025dynamic}.
For polarization in particular, the work most closely aligned with ours is \citet{Faliszewski_Kaczmarczyk_Sornat_Szufa_Wąs_2023}, where diversity and polarization indices are defined for elections with complete rankings in terms of clusterings that minimize the ballots' summed $d_H$-distances to $k$ centers.
Relatedly, \citet{Faliszewski_Mertlova_Nunn_Szufa_Was_2025} develop distances between elections of different sizes with partial ballots and use them to extend the map-of-elections framework.
Both works use Kendall-based distances.
Our tools likewise accommodate partial ballots, while providing parallel Kendall and Spearman versions of the indices that can be computed at scale.
\citet{Busse_Orbanz_Buhmann_2007} cluster  ballots  by fitting the parameters of a Mallows model; by contrast, we identify clusters without assuming that the ballots follow any particular distribution.

\section{Ballot graphs and coordinate embeddings}\label{S:ballotgraph}\label{S:coords}

Let $\Omega_m$ be the set of possible partial or complete ballots in an election with $m$ candidates. The candidates can be denoted $\{A,B,\dots \}$ or $\{A_1,\dots,A_m\}$.  We will adopt flexible notation for ballots, so that 
$$(A,C,D,B),\quad  A>C>D>B,\quad  ACDB,\quad  \hbox{\rm and} \quad 
\begin{psmallmatrix}A\\ C\\ D\\ B\end{psmallmatrix}$$ 
all denote the same ballot.  We will record partial ballots as having unmentioned candidates tied at the end of the ballot, so that 
$(A,B)$ or $(A,B,-,-)$ is identified with $(A,B,\{C,D\})$.
Partial ballots therefore only allow ties at the end of the ballot;
 {\em generalized ballots} would allow ties in arbitrary positions.

\subsection{Coordinate embeddings}\label{subsec:coord}

\begin{definition}[\textbf{Borda and head-to-head distance}]
Suppose $\sigma$ is a permutation of the symbols $1,\dots,m$, 
with the convention that $\sigma$ maps a candidate index to the rank of that candidate on the ballot, 
i.e., for candidates $A_1,\dots,A_m$, we have that $\sigma(i)$ is the position in which the voter ranked $A_i$.
In this notation, the ranking corresponding to $\sigma$ is 
$\begin{psmallmatrix}
    A_{\sigma^{-1}(1)}\\ \vdots\\ A_{\sigma^{-1}(m)}\end{psmallmatrix}.$ 

Extend this to partial ballots 
by the \emph{pessimistic}
convention: if $j$ is an unlisted candidate, then $\sigma(j)=m$. 
In places, we will also refer to the \emph{averaged} convention whereby $\sigma(j)$ equals the average of the rankings that all of the unlisted candidates would have received if they had been ranked.  Averaging maintains the total sum of ranks for any ballot, complete or partial, as $\sum_i \sigma(i)=\binom{m+1}{2}$, while the pessimistic convention lets that sum vary---in other words, embedding with the averaged convention sends all ballots onto a common hyperplane.

Define functions $\bb:\Omega_m\to \R^m$ and 
$\hh: \Omega_m\to\R^{\binom m2}$ by
$$\bb(\sigma)(i)=m-\sigma(i), \qquad 
\hh(\sigma)(\{i,j\})= \begin{cases}
    +1 & \sigma(i)<\sigma(j)\\
    0 & \sigma(i)=\sigma(j) \\
    -1 & \sigma(i)>\sigma(j)
\end{cases} \,\,\,(\text{assuming }i<j).$$
Then the Borda distance and head-to-head distance are defined by
$$d_B(\sigma,\tau)=\frac 12 \|\bb(\sigma)-\bb(\tau)\|_1, \quad
d_H(\sigma,\tau)=\frac 12 \|\hh(\sigma)-\hh(\tau)\|_1.$$
\end{definition}


\begin{example}
In $\Omega_4$, the ballot $AD$ is  mapped to the Borda point 
$(3,0, 0, 2)$.  Since the $6$ head-to-head comparisons are $AB,AC,AD,BC,BD,CD$, the ballot $AD$ has the head-to-head image $(1,1,1,0,-1,-1)$.  
\end{example}

The Borda distance $d_B$ can be visualized through positional shifts, as illustrated in Figure~\ref{F:shifts}.

\begin{figure}[bht!]\centering
\includegraphics[width=3in]{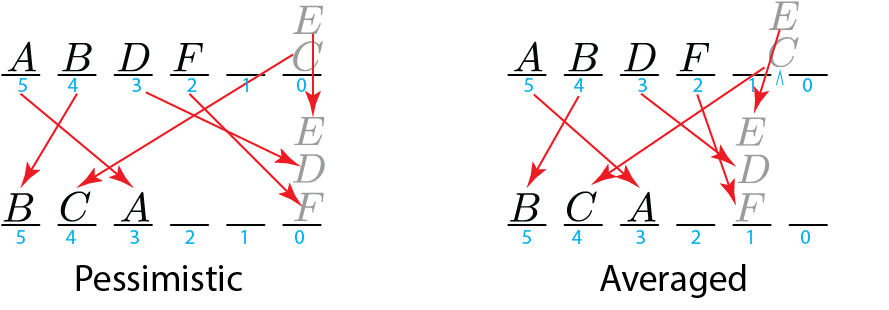}
\caption{Summing the horizontal displacement represented by the  arrows gives the total number of immediate positional shifts required to convert the ballot $(A,B,D,F)$ into $(B,C,A)$, which equals twice the Borda distance between these ballots.  Under the pessimistic convention, the total shifts sum to 12 (so $d_B = 6$), while they sum to 10 in the averaged convention (so $d_B=5$).  Here the unmentioned candidates are positioned so that shift counting exactly corresponds to the definition of $d_B$.}\label{F:shifts}
\end{figure}

The head-to-head embedding uses an order on candidates, whether alphabetical or an arbitrary fixed order, and the value $+1$ means that the first-indexed candidate is ranked higher.
For complete ballots, $d_H$ is exactly the Kendall tau distance, counting neighbor swaps between permutations;  for short ballots, $d_H$ agrees with the generalization of the Kendall tau metric discussed in~\cite{Fagin_Kumar_Mahdian_Sivakumar_Vee_2006}.  

\subsection{Ballot graphs}
Any finite metric space can be given a graph realization in a trivial way: by adding an edge between each pair $x,y$ of vertices with weight $d(x,y)$.  
(The triangle inequality ensures that the path metric on that graph is $d$.)
Thus, for a graph realization to add value, it should be relatively sparse in terms of the edges needed to recover the metric as the path metric.  In this section we introduce
{\em ballot graphs} that are sparse in the sense that, though the number of nodes 
will grow as roughly $2m!$ for the graph on $m$ candidates, the degree of each vertex will be only polynomial in $m$.

\begin{definition}[\textbf{Ballot graphs}]
The {\em ballot graph} $\G_m$ is a weighted, undirected graph with possible ballots as the vertex set ($V=\Omega_m$).  
Since we treat unmentioned candidates as being tied at the end of a ballot, we  identify each ballot of length $m-1$ with its extension to a complete ballot.
There are two types of edges.
\begin{itemize}
\item Neighbor swaps: $(\sigma,\tau)\in E$ whenever $\tau$ is obtained from $\sigma$ by swapping the candidates in two adjacent ranked positions. These edges are given unit weight.
\item Truncation and extension: 
when $\sigma$ and $\tau$ have length $k-1$ and $k$ and agree in the first $k-1$ positions, then 
$(\sigma,\tau)\in E$ 
is given weight $\frac{m-k}2$.  
\end{itemize}
The {\em shortcut ballot graph} $\mathcal G_m^+$ adds a more general swap edge.
\begin{itemize}
    \item General swaps: $(\sigma,\tau)\in E$ whenever $\tau$ is obtained from $\sigma$ by swapping the candidates in ranked positions $i$ and $i+k$ for some $i,k$, and its weight is set to $k$. In other words, these edges correspond to swaps of candidates in not-necessarily-successive positions, at a cost that is equal to the difference in rank. 
\end{itemize}
\end{definition}

Note that neighbor swaps are a special case of general swaps.
We will use the notions of edge weight, edge length, and the "cost" of a move interchangeably.  Treating edge weights as lengths lets us study the induced path metric:  the distance between two ballots is equal to the length of the shortest edge-path between them.

\begin{example}[Illustrating ballot graph construction] \

\begin{itemize}
    \item The ballot graph $\G_3$ is shown in Figure~\ref{fig:G3}.  For instance, the distance from $ABC$ to $ACB$ is one because they differ by a neighbor swap.
\item In $\G_4$ (Figure~\ref{fig:G4}), the distance from $ABCD$ to its reversal $DCBA$ equals $6$.  Note that this can either be realized through a path of 6 neighbor swaps or through a path that goes through bullet votes:
$$ABCD\rightarrow ABC\rightarrow AB\rightarrow A \rightarrow
 AD \rightarrow DA\rightarrow D \rightarrow DC \rightarrow DCB \rightarrow DCBA.$$
 (That these two kinds of paths have equal length generalizes to connecting any complete ballot and its reversal in any $\G_m$.)  In $\G_4^+$, the distance drops to $4$ with the use of a shortcut edge $ABCD\to DBCA$, whose length is 3, plus one additional neighbor-swap.
\item In $\G_6$, the edge from $(A,B)$ to $(A,B,C)$ has weight $\frac{3}{2}$.  
\end{itemize}
\end{example}

\begin{figure}[bht!]\centering
\begin{tikzpicture}
\node at (0,0) {\includegraphics[width=2.8in]{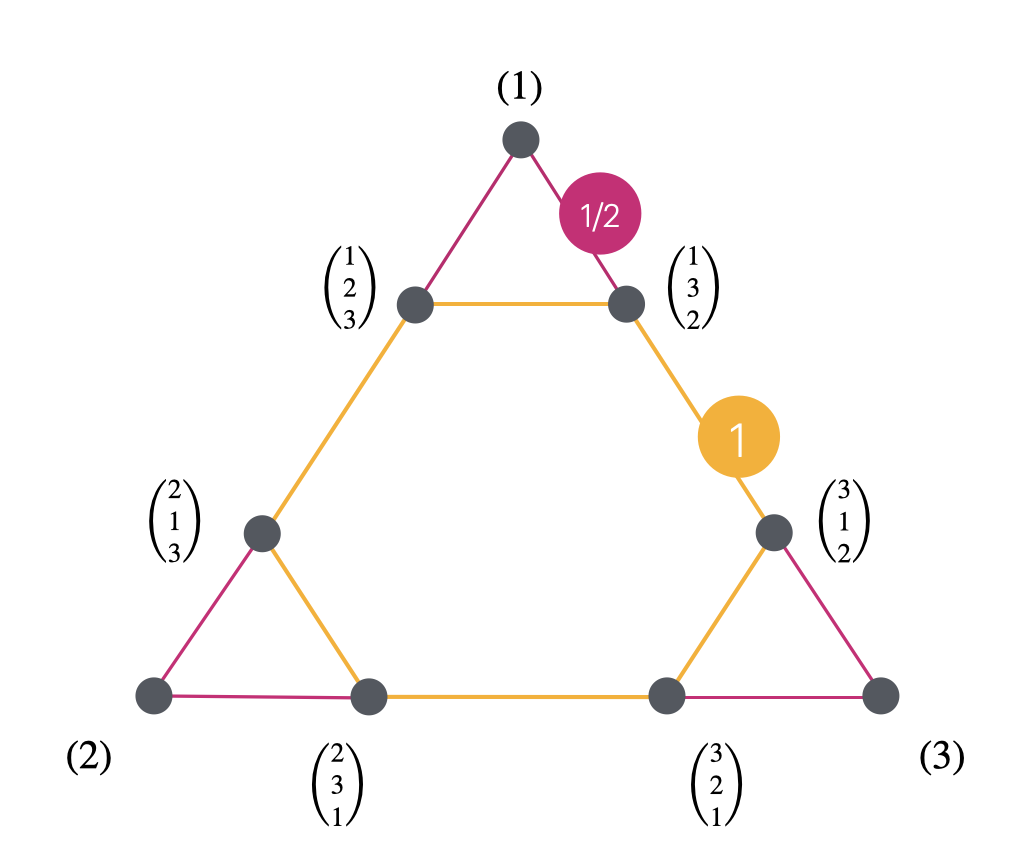}};
\node at (7,0) {\includegraphics[width=2.8in]{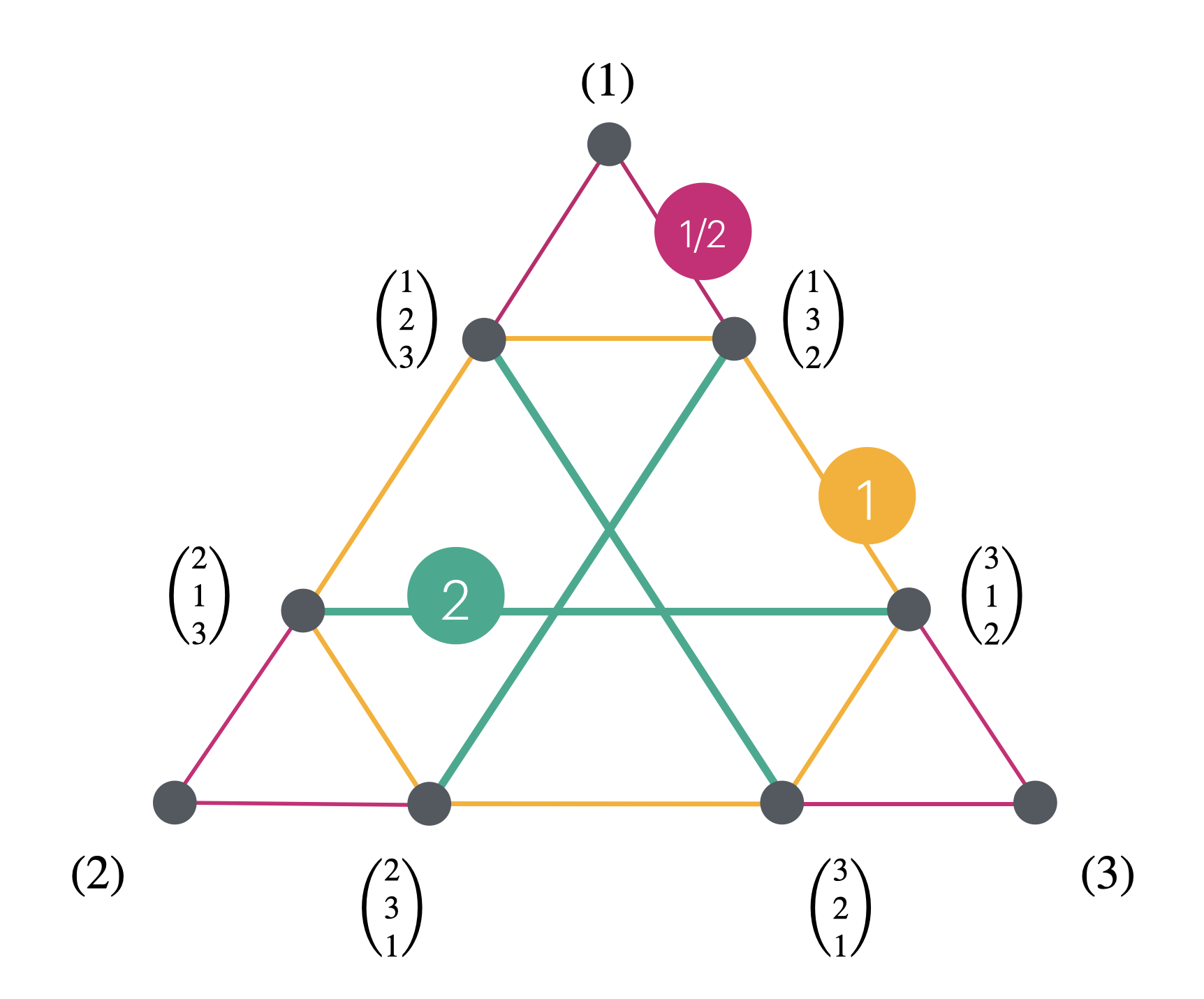}};
\end{tikzpicture}

\caption{The basic ballot graph $\G_3$ and shortcut ballot graph $\G_3^+$.  Swaps of the first and last place correspond to new edges of length 2, providing shortcuts between nodes that would otherwise be 3 apart.}\label{fig:G3}
\end{figure}

\begin{figure}[bht!]\centering
\begin{tikzpicture}
\node at (0,0) {\includegraphics[width=5in]{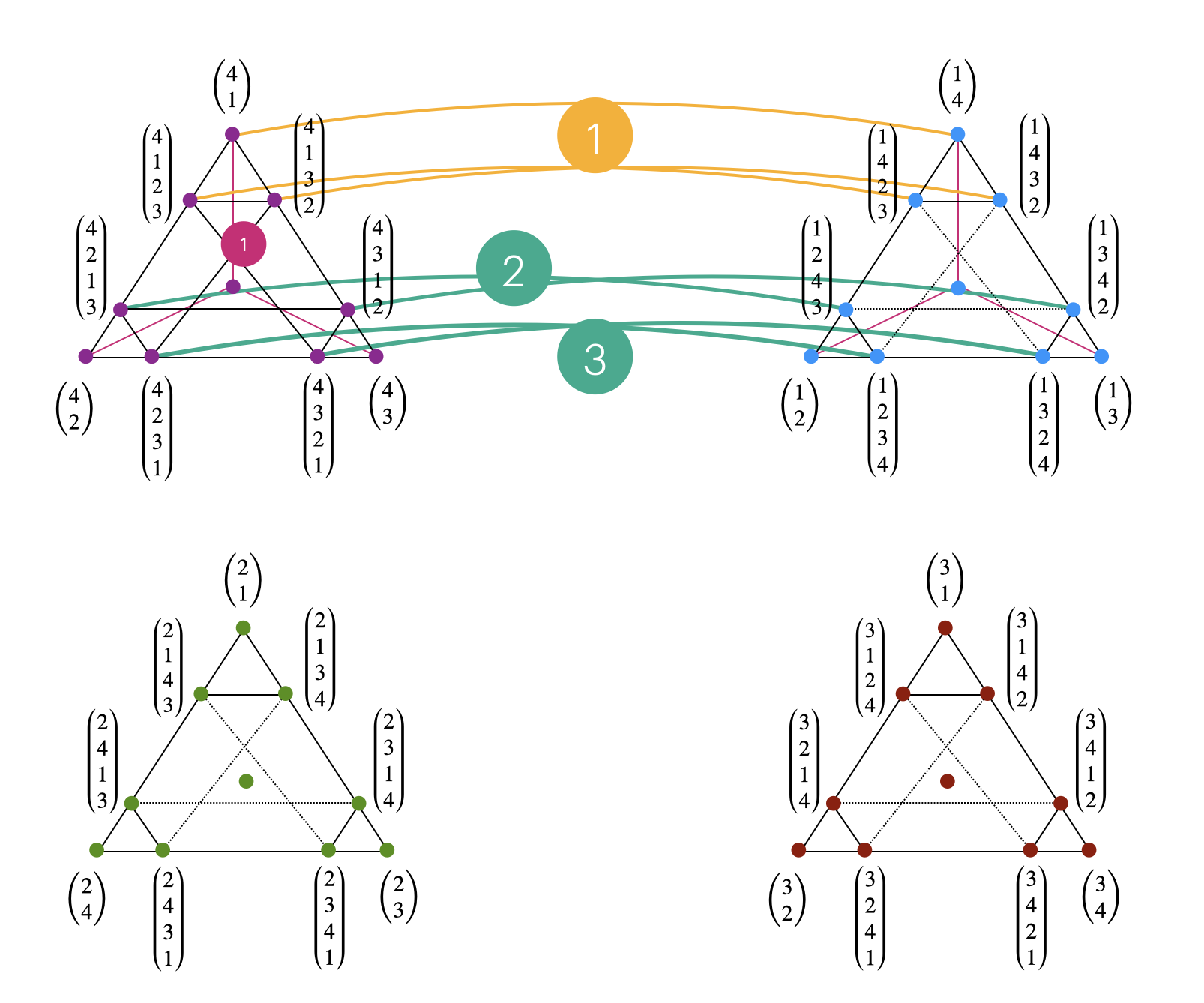}};

\end{tikzpicture}

\caption{An illustrative portion of the shortcut ballot graph $\G_4^+$, showing the connections between ballots headed by candidates 4 and 1.  The construction of this picture shows the recursion $|\Omega_m|=m\cdot|\Omega_{m-1}|+m$. }\label{fig:G4}
\end{figure}

\begin{proposition}
In either $\G_m$ or $\G_m^+$,  size of the vertex set is $O(m!)$ and the degree of each vertex is $O(m^2)$.  
\end{proposition}

\begin{proof}
The number of complete or partial rankings of $m$ objects (including the empty ranking and with no identifications) is $\lfloor em!\rfloor$.  This can be seen by counting partial rankings as $\sum_{k=0}^m \frac{m!}{(m-k)!}=m!\sum_{j=0}^m \frac{1}{j!}$ and comparing to $e=\sum_{j=0}^\infty \frac{1}{j!}$.  

Because we identify ballots of length $m$ with their truncations of length $m-1$, the number of vertices grows like $(e-1)m!$.  If a vertex represents a ballot of length $1<k<m-1$, it has one  neighbor by truncation, $m-k$  neighbors by extension, and $k-1$ swap neighbors, giving degree $m$ in $\G_m$.  Complete ballots have one neighbor by truncation (of length $m-2$) and $m-1$ swap neighbors, for the same degree.  Bullet votes have $m-1$ extensions; if we include the empty ballot in the graph, their degree is also $m$ (making $\G_m$ an $m$-regular graph), and otherwise it is $m-1$.

In the shortcut graph, the number of swap neighbors for the ballot of length $k$ is $\binom k2$ instead of $k-1$, which contributes less than $\frac 12 m^2$ to the degree.
\end{proof}

These graphs extend important mathematical objects that occur in group theory and combinatorics.  Restricting $\G_m$ to the complete ballots recovers the Cayley graph of the symmetric group $S_m$ (with standard transposition generators), which shows up in Figure~\ref{fig:G3} as a $6$-cycle in $\G_3$.  Adding shortcut edges amounts to adding additional generators.  
Another way to phrase this is that the 1-skeleton of the permutahedron---a polytope used to track the structure of permutations---sits in each $\G_m$. Using this kind of mathematical structure allows us to phrase the results of elections as distributions on a graph or network.\footnote{There are many deep open questions about these:  for instance, the Babai conjectures relate to diameter bounds and spectral gaps for Cayley graphs of $S_m$ as the generators vary.}

\subsection{Generalized ballot graphs}\label{sec:gen_graph}
Next, consider the case that ballots are allowed to include general ties.  For example, with $7$ candidates, the \emph{generalized ballot} 
$$ (C,\{A,D\},\{B,F,G\},E)$$
reflects  equal preference for $A$ and $D$, and for $B$, $F$, and $G$.
This allows us to write a preference between slates:  if 
$\mathcal C=\mathcal A \sqcup \mathcal B$ is a partition of the set $\mathcal C$ of candidates into slates, 
then the preference $\mathcal A>\mathcal B$ 
can be considered as a generalized ballot.
Without loss of generality, every candidate appears once in a generalized ballot;
complete and partial ballots are both special cases, since partial ballots were already interpreted as having a tie in the last position.
Let $\WO_m$ denote the set of generalized ballots in an election with $m$ candidates.  

The \emph{generalized ballot graph} $\WG_m$ is a graph on these nodes. There is an edge between two nodes whenever one ballot is obtained from the other by merging the  candidate sets in adjacent positions, and the length of that edge is  $\frac 12$ times the product of their orders.  This simple edge rule correctly generalizes both edge rules for $\G_m$. 

\begin{example}\label{E:yaya}
The graph $\WG_4$ contains the unit-length path
$$(A,B,C,D)\stackrel{.5}{\rightarrow}(A,\{B,C\},D)\stackrel{.5}{\rightarrow}(A,C,B,D).$$
Swapping  any pair of adjacent singleton candidates similarly has cost 1.  
\end{example}

We enlarge the domain of $\hh$ to 
$\WG_m$ accordingly:  as before, its ${\binom m2}$ coordinates represent unordered pairs of candidates, and the value of each coordinate is $-1$ (lose), $0$ (tie), or $1$ (win).  Here a tie means that the candidates belong to the same candidate set.  

Note that these generalized ballot graphs have a super-factorial number of nodes (counted by the Fubini numbers) and now have exponentially growing degree. That is because each generalized ballot has at most $m-1$ ways to merge neighboring candidate sets, but if a candidate set has $k$ elements then there are $2^k$ ways to split it, so for instance a bullet vote has many neighbors.

\subsection{Relating the ballot graphs to coordinate embeddings}

\begin{theorem}[Graphs match $L^1$ distances]\label{T:L1}
The distance $d_H$ is realized as the path metric on the basic ballot graph $\G_m$, while $d_B$ is  the path metric on the shortcut ballot graph $\G_m^+$.  
\end{theorem}

To prove this, it is convenient to introduce notions of strong and weak disagreement.
We say that ballots $\sigma,\tau$ have a strong disagreement on candidates $i,j$ if one candidate is ranked higher on $\sigma$ while the other is ranked higher on $\tau$; a weak disagreement occurs when one ballot expresses a preference where the other ballot records a tie.  Then for a pair of ballots $\sigma,\tau$ we can define 
$\str(\sigma,\tau)$ and $\wk(\sigma,\tau)$ as the total number of strong and weak disagreements, respectively.  
For example, if $\sigma=(A,B,C)$ and $\tau=(A,E)$ in $\Omega_5$, then $\str(\sigma,\tau)=2$ and $\wk(\sigma,\tau)=4$.  
Since a strong disagreement occurs when a coordinate of the head-to-head vector is $+1$ in one ballot and $-1$ in the other, it contributes 2 to the $L^1$ difference; similarly, a weak disagreement contributes 1.  It is therefore immediate that 
$$d_H=\str+\frac 12 \wk$$
for any pair of ballots.  To prove that the graph distance matches $d_H$, it suffices to describe graph distance in terms of strong and weak disagreements.

\begin{proof} We begin with $d_H$.  For a pair of ballots $\sigma,\tau$ that are adjacent in $\G_m$, we verify that the weight of the edge between them equals 
$w(\sigma, \tau) =\str+\frac 12 \wk = d_H(\sigma,\tau)$ by checking the cases of neighbor-swap and extension.  Now  consider a geodesic in $\G_m$ visiting vertices $\sigma_1,\dots,\sigma_r$ and note that 
the graph distance equals $\sum w(\sigma_i,\sigma_{i+1})$ by definition of geodesic.  On the other hand, $d_H(\sigma_1,\sigma_r)$ is less than or equal to this sum, because $d_H$ is defined as an $L^1$ difference, so it satisfies the triangle inequality.  Thus $d_H$ is less than or equal to  the graph distance.

To show that this is an equality, we construct a path realizing the distance equality between arbitrary ballots $\sigma,\tau$.
To do this, we construct "waypoints":  if $\sigma$ is a partial ballot, let $\sigma'$ be the possibly longer ballot which agrees with $\sigma$ on all listed candidates, then ranks candidates that were mentioned in $\tau$ but not $\sigma$ in the order in which they appear in $\tau$; likewise, $\tau'$ starts like $\tau$ but adds missing candidates in the order of appearance in $\sigma$.  

Consider the path from $\sigma$ to $\sigma'$ via extensions, continuing to $\tau'$ by a minimal sequence of neighbor swaps, then extending to $\tau$ by truncation.
To see that this suffices, note that for two distinct ballots that are complete (or more generally that have the same candidate set), there must be at least one neighboring pair of candidates that witnesses the strong disagreement.  Otherwise, every $X$ ranked above $Y$ by the first ballot would also be ranked above $Y$ by the second ballot, forcing the ballots to be identical.  This means that there is always a neighbor swap that reduces the strong disagreement count by $1$, until that count is zero.  Thus the length of the path $\sigma'\rightarrow\tau'$ is $\str(\sigma,\tau)$.  

Suppose there are $k$ candidates mentioned in $\tau$ but not $\sigma$ and $l$ candidates mentioned in neither ballot.
Then the path distance from $\sigma$ to $\sigma'$
is  $\frac{l+\dots+(l+k-1)}{2}= \frac 12 \left(\binom{k}{2}+kl\right)$.
The number of weak disagreements involving those candidates is $\binom{k}{2} + kl$.  Similarly between $\tau$ and $\tau'$ the total edge weight is one-half of the number of weak disagreements involving candidates unmentioned in $\tau$ but mentioned in $\sigma$.  
Taken together, we've shown that the natural path $\sigma\to\sigma'\to\tau'\to\tau$ has total edge weight $\str(\sigma,\tau)+\frac 12 \wk(\sigma,\tau)$, as desired.  

Next, to prove that $d_B(\sigma,\tau)$ equals the path distance in $\G^+_m$, we observe that a path in $\G^+_m$ is a $d_B$-geodesic if and only if each candidate's Borda score changes monotonically along the path between its starting score in $\sigma$ and its ending score in $\tau$.  We will use the same waypoints $\sigma',\tau'$.  It is straightforward to see that the above-constructed path segments $\sigma\rightarrow\sigma'$ and $\tau\rightarrow\tau'$ are $d_B$-geodesics.  So it remains to construct a $d_B$-geodesic $\sigma'\rightarrow\tau'$ in $\G_m^+$ and verify that the composition of paths $\sigma\rightarrow\sigma'\rightarrow\tau'\rightarrow\tau$ is a $d_B$ geodesic. 

From its position in $\sigma'$, every candidate needs to shift either to a lower position ($\downarrow$), higher ($\uparrow$) or neither ($\square$) to reach its target position in $\tau'$.  We represent this as a {\em state vector} as in Figure~\ref{fig:statevector}.
We proceed in steps, each time swapping the candidates corresponding to the first $\uparrow$ symbol and the last $\downarrow$ symbol before it.   Repeating this process until reaching $\tau'$ produces a Borda geodesic because each candidate can only move toward its target position; the move never overshoots the target position because the candidates in between those who swapped were already in their final positions.  Moreover, the composite path $\sigma\rightarrow\sigma'\rightarrow\tau'\rightarrow\tau$ is a $d_B$-geodesic because even the candidates who move into and out of ties have Borda scores that make monotonic progress along the path.  Monotonicity means that differences telescope
($a\le b\le c \le d \implies 
|a-d|=|a-b|+|b-c|+|c-d|$), so the concatenation is geodesic, as required.
\end{proof}

\begin{figure}[bht!]\centering
\includegraphics[width=3in]{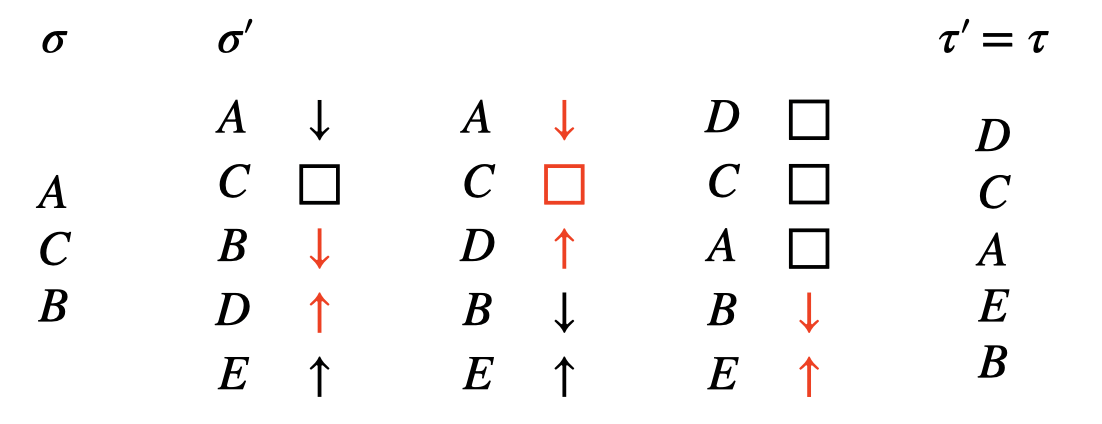}
\caption{For ballots
$\sigma=(A,C,B)$ and $\tau=(D,C,A,E,B)$, we first complete $\sigma$ to $\sigma'$ and then make swaps as indicated (in red) for a total Borda distance of $d_B(\sigma',\tau')=1+2+1=4$. This correctly matches half the $L^1$ difference of $\bb(\sigma')=(4,2,3,1,0)$ and $\bb(\tau')=(2,0,3,4,1)$.}
\label{fig:statevector}
\end{figure}

\begin{remark}
This proof fails for the averaged Borda convention.  With the pessimistic convention, moving a candidate from $\sigma$ out of a tie at the end does not change the position of the remaining unmentioned candidates, so it moves $\sigma$ closer to $\sigma'$.  But under averaging, the extension move can improve the position of one candidate but worsen the position of others.

In fact, a simple example of this kind suffices to show that our local moves are not enough to realize averaged Borda with graph distance.  Consider what is needed to transform a bullet vote for $A$ to a bullet vote for $B$ when there are many candidates.  In the averaged convention, all other candidates have the same Borda score in the initial and terminal ballot.  But extending $A$ to $AB$ changes the average score of the others, and we easily see that a direct edge from $A$ to $B$ would be needed.  Similar arguments force an explosion of degree.
\end{remark}

The following is a simple application of the strong and weak disagreement count.
\begin{proposition}\label{C:resolve}
Let $\widehat{\sigma}$ and $\widehat{\tau}$ be obtained by using the same uniformly random ordering of the candidates to resolve all ties in $\sigma$ and $\tau$. Then $d_H(\sigma,\tau)$ equals the expected number of adjacent swaps needed to convert $\widehat{\sigma}$ into $\widehat{\tau}$.
\end{proposition}
\begin{proof}
A strong disagreement remains a strong disagreement. A weak disagreement is resolved into either an agreement or a strong disagreement with equal likelihood, while a pair tied in both ballots is resolved identically in both completions. Thus, the expected number of strong disagreements between the completions is
$\str+\frac12\wk=d_H(\sigma,\tau)$.
\end{proof}

In the following Theorem and its two Corollaries, we extend the classical result of~\cite{Diaconis_Graham_1977} from the setting of complete rankings to our setting that includes partial ballots.   

\begin{theorem}\label{T:stwkB}
$$\frac 12 \str+\frac 12 \wk \le d_B 
\le \str + \frac 12 \wk = d_H,$$
with equality on the left iff $\str=0$.
\end{theorem}
\begin{proof}
The proof of Theorem~\ref{T:L1} established that
$$ d_B(\sigma,\tau) = 
\underbrace{d_B(\sigma,\sigma') + d_B(\tau,\tau')}_{=\frac 12 \wk(\sigma,\tau)}
+ d_B(\sigma',\tau'),$$
so it remains to show that $d_B(\sigma',\tau')$ (which is the $\G_m^+$-path distance between $\sigma'$ and $\tau'$) is at least half $\str(\sigma',\tau')$ (which is the $\G_m$-path distance), with strict inequality when $\str>0$. For this, it will suffice to prove that the weight of each shortcut edge of $\G_m^+$ is more than half the $\G_m$-distance between its endpoints.  Every shortcut edge of length $k$ represents two candidates jumping over $k-1$ stationary candidates located in between them;  with neighbor swaps, this can be realized by moving the lower candidate $k$ spots and the higher candidate $k-1$ spots, for $2k-1$ overall.  Thus each shortcut speeds up a path segment by a factor less than two.  If $\sigma'=\tau'$ both sides are zero; otherwise every $\G_m^+$-geodesic from $\sigma'$ to $\tau'$ uses at least one edge, so the inequality is strict.
\end{proof}

\begin{corollary}\label{C:hihi}
For distinct ballots, $d_B\leq d_H < 2d_B$.
\end{corollary}

\citet{Fagin_Kumar_Mahdian_Sivakumar_Vee_2006} define a family of metrics they denote by $K^{(p)}$ as $\str+p\cdot \wk$, so that 
$d_H=K^{(1/2)}$.  They prove that these are true metrics on the space of ballots with arbitrary ties when $p\in [1/2,1]$, and that they are "near-metrics" (satisfying a relaxed triangle inequality) for $p\in (0,1/2)$.  

\begin{remark}
The Borda distance $d_B$ is not a multiple of any $K^{(p)}$.  To see this, note that complete ballots have no weak disagreements.  
So for $d_B$ to be a rescaling of some $K^{(p)}$ on complete ballots, it would have to be proportional to the number of strong disagreements.
Consider $\sigma=ABCDE$, $\tau=EBCDA$, and $\pi=DBCAE$.  The candidate swap with rank difference four means ballot $\tau$ requires 7 inversions to convert to $\sigma$ while their shortcut edge has length $4$; the swap three units apart means $\pi$ vs. $\sigma$ needs 5 inversions while their shortcut edge has length $3$.  
This shows that $d_B$ is not a fixed multiple of $\str$, so 
$d_B\neq a\!\cdot\! \str+b\!\cdot\! \wk$ for any $a,b$.
\end{remark}

The next result is shown by Fagin et al. for averaged Borda, and also holds for our (pessimistic) $d_B$.

\begin{corollary}\label{P:my_diaconis} For any pair $\sigma,\tau\in\Omega_m$ of distinct ballots,
$\displaystyle\frac{d_H(\sigma,\tau)}{d_B(\sigma,\tau)}\in[1,2)$,
and these bounds are sharp.
\end{corollary}
\begin{proof}
The bound is just a rephrasing of Corollary~\ref{C:hihi}, so it remains to establish optimality.
If $\sigma\in\Omega_m$ is a complete ballot, let 
$\sigmarev$ be its reversal, with candidates listed in opposite order.  We have $d_H(\sigma,\sigmarev)=\binom{m}{2}$ because every pair will have a strong disagreement.  
On the other hand, reversal sends the candidate in position $j$ to position $m+1-j$, a displacement of $|m+1-2j|$, so $d_B(\sigma,\sigmarev)= \tfrac12\sum_{j=1}^m |m+1-2j| = \lfloor m^2/4\rfloor$, and comparing these gives a sequence with 
$\displaystyle\frac{d_H(\sigma,\hat\sigma)}{d_B(\sigma,\sigmarev)}
\to 2$.  For the other bound, just note that any pair of ballots connected by a single neighbor swap realizes $d_H=d_B$.
\end{proof}

In the proof of Theorem~\ref{T:stwkB}, we noted that a shortcut edge of length $k$ replaces a path of length $2k-1$ in the basic ballot graph.  This means its distance savings is $k-1$.  Thus if the shortcut edges used in a geodesic have lengths $k_1,k_2,\dots,k_r$, we have $d_H-d_B \le (\sum k_i)-r$, with equality if the expanded path is geodesic.

For complete ballots in $S_m$, Diaconis and Graham show that the expected Kendall tau distance to the identity is $I\sim m^2/4$ and the expected Spearman footrule distance is $D\sim m^2/3$ as $m\to\infty$ \cite[Table 1]{Diaconis_Graham_1977}.  Since Diaconis and Graham's $I$ is our $d_H$, while their $D$ is twice our $d_B$, the ratio of the average value of $d_H$ to the average value of $d_B$ tends to $3/2$.

As a final remark on natural ways to metrize the space of ballots, note that we can regard a ballot with ties as a collection of all complete ballots obtained by arbitrary resolution of ties.  
Rather than computing the average as in Proposition~\ref{C:resolve}, one could measure the maximum or the minimum distance between points in the resolution of one ballot to those representing the other.  In \cite{Critchlow_1985}, Critchlow also studied the  Hausdorff distance between these point clouds $\bar\sigma$ and $\bar\tau$, namely the smallest $\epsilon$ so that the $\epsilon$-neighborhood of $\bar\sigma$ contains $\bar\tau$, and likewise the $\epsilon$-neighborhood of $\bar\tau$ contains $\bar\sigma$, with respect to swap distance.   \citet{Fagin_Kumar_Mahdian_Sivakumar_Vee_2006} proved the
following characterization.  Let $\sigma,\tau$ be partial ballots and define $\wk(\sigma\to\tau)$ to be the number of candidate pairs tied in $\sigma$ but strictly ordered in $\tau$, while $\wk(\tau\to\sigma)$ is defined with the roles reversed. Thus $\wk(\sigma,\tau)$ is the sum of these two asymmetric values. Then
$$
\dHaus(\sigma,\tau) = \str(\sigma,\tau) + 
\max\left(
\wk(\sigma\to\tau),\wk(\tau\to\sigma)
\right).
$$
Because the maximum of two numbers is no greater than their sum but no less than half of their sum, we get $d_H\le \dHaus \le 2d_B$.

\subsection{Relating the generalized ballot graph to coordinate embeddings}
Between a pair of generalized ballots, we define weak and strong disagreements exactly as before; in particular, a pair of ballots has a weak disagreement between two candidates if the candidates tie (are in the same ranking position) with respect to exactly one of the two ballots.  Just like before, we define $$d_H(\alpha,\beta) = \frac 12 \|\hh(\alpha)-\hh(\beta)\|_1$$ for all $\alpha,\beta\in\WO_m$, and we observe that
$$d_H(\alpha,\beta) = \str(\alpha,\beta) + \frac 12 \wk(\alpha,\beta).$$  

Up to a factor of $2$, this is the distance on weak orders introduced in
\citet{Kemeny_1959} and characterized axiomatically in
\citet{kemeny-snell1962}; among their axioms is a betweenness condition,
$d(\alpha,\gamma) = d(\alpha,\beta) + d(\beta,\gamma)$ whenever $\beta$
lies between $\alpha$ and $\gamma$, which is exactly the additivity that
a graph realization makes explicit. The 1959 paper includes an image
of the graph $\WG_3$, but no construction for general $m$.
Thus, Kemeny's work already anticipates the idea that disagreement distance on weak orders can be realized by elementary moves
through a graph of rankings.  The following proposition, which generalizes a portion of Theorem~\ref{T:L1}, gives the corresponding statement for our particular weighted generalized
ballot graph.

\begin{proposition} The path metric on  $\WG_m$ realizes $d_H$ on generalized ballots.
\end{proposition}

\begin{proof}
To start, let's define the common refinement of two generalized ballots.  It is the coarsest grouping of candidates so that two candidates appear together in the common refinement iff they are grouped together in both ballots.  For instance, the common refinement of $(A,\{B,C\},D,\{E,F\})$ and 
$(B,\{A,C\},\{D,E,F\})$ consists of $A,B,C,D$ singletons and the pair $\{E,F\}$.  

Next, note that the edges in $\WG$ are defined as merges of candidate sets ranked in neighboring positions.  Recall that if one set has $r$ candidates and the neighboring set has $s$ candidates, then the edge between them was given length $rs/2$.  Also note that the ranking that had them separate had $rs$ strong preferences among those candidates, while the ranking that has them all grouped has only weak preferences.

For a pair $\alpha,\beta \in\WO_m$, we must construct a path from $\alpha$ to $\beta$ in $\WG_m$ whose edge weights sum to the weighted disagreement count.  
We design such a path with waypoints $$\alpha\rightarrow\alpha'\rightarrow\beta'\rightarrow \beta,$$ where $\alpha'$ and $\beta'$ resolve all ties from their respective ballots that are not present in the common refinement.  To resolve such a tie in $\alpha$, we use the strict preference seen in $\beta$, and vice versa.  Thus $\alpha'$ and $\beta'$ both have the common refinement as their candidate sets.  
Each edge along the paths $\alpha\to\alpha'$ and $\beta\to\beta'$ splits a candidate set into two adjacent sets of sizes $r$ and $s$. Such an edge has length $rs/2$ and resolves exactly $rs$ weak disagreements between $\alpha$ and $\beta$. Since these paths collectively resolve all and only the weak disagreements, their lengths sum to $\frac12\wk(\alpha,\beta)$.

On the other hand, the path $\alpha'\to\beta'$ consists of swaps of adjacent candidate-sets, where each swap is performed by merging the two sets and then splitting them in the opposite order. If the sets have sizes $r$ and $s$, the two edges each have length $rs/2$, so the swap has total length $rs$ and reverses exactly $rs$ strong disagreements. Although the intermediate merged ballot introduces weak disagreements, these disappear upon splitting. Choosing adjacent swaps without backtracking therefore gives a path of total length $\str(\alpha,\beta)$.
\end{proof}

The shortcut constructions for $d_B$ and $d_{B_\text{avg}}$ (pessimistic and averaged conventions, respectively) generalize cleanly to $\WO_m$, with the Borda conventions applied within each candidate set in the corresponding way. 

\citet{Fagin_Kumar_Mahdian_Sivakumar_Vee_2006} prove that the previous distance bounds extend to generalized ballots  under the averaged convention.
\begin{proposition}\label{P:complete2} If $\alpha, \beta\in\WO_m$ are distinct, then 
$
\displaystyle\frac{d_{H}(\alpha,\beta)}{d_{B_\text{avg}}(\alpha,\beta)}\in[1,2]$.
\end{proposition}

However, the corresponding statement is false for $d_B$.
\begin{example}
If $\alpha=(\{A_1,...,A_k,{\color{red}X}\},\{B,{\color{red}Y}\})$ and $\beta=(\{A_1,...,A_k,{\color{red}Y}\},\{B,{\color{red}X}\})$ in $\WO_{k+3}$, then we have $d_B(\alpha,\beta) = 2$ and $d_{H}(\alpha,\beta) = 2+k$, so $\lim\limits_{k\to\infty}\displaystyle\frac{d_{B}(\alpha,\beta)}{d_{H}(\alpha,\beta)} = 0$.  
\end{example}

\FloatBarrier
\section{Finding blocs and slates; synthetic validation}
\label{sec:coordinate-clusters}

\subsection{Grouping the ballots}\label{subsec:opti}

Once the ballots are given coordinate representations by the Borda ($\bb$) or head-to-head ($\hh$) embedding, we can employ standard clustering techniques to partition the ballots.  Recall that the Kemeny problem is to find the optimal center of a profile by $d_H$:  that is, the (complete) ranking minimizing the summed Kendall tau distance to voters.  
An algorithm that outputs the minimizer is known as the Kemeny voting rule; this is well known to be NP-hard.  Similarly, the 2-Kemeny problem seeks two centers that optimize the total cost, given by summing the distance of voters to the nearer center. Previous authors have recognized the descriptive relevance of these centers.  Indeed, \citet{Faliszewski_Kaczmarczyk_Sornat_Szufa_Wąs_2023} {\em define} a polarized election as one better described by two centers than one. 

We can make nontrivial use of the metric results to observe that sufficiently polarized elections have canonical clusterings.  We work in the setting of a general metric space $(X,d)$.  Recall that $d$ and $d'$ are called 2-biLipschitz equivalent metrics on the same set $X$ if $\frac 12 d'(x,y)\le d(x,y) \le 2 d'(x,y)$ for all $x,y$.  Given a set of cast ballots and a set of $k$ centers in $X$, the associated \emph{clustering} is the partition obtained by assigning each ballot to its nearest center.  A set of centers is \emph{minimax-optimal} if it minimizes the largest distance from a ballot to its nearest center, and \emph{sum-optimal} if it minimizes the sum of these distances; the summed objective is used throughout this paper, but the minimax objective also appears elsewhere in social choice theory under that name or via ``$k$-center optimization.''

Write $B_r(x)$ for the ball of radius $r$ about $x$.
Call a profile $(R,r)$--\emph{polarized} with respect to $d$ if it has sufficient separation and concentration: there are centers $x,y$ with $d(x,y)>R$ such that $B_r(x)$ and $B_r(y)$ together contain all cast ballots, and each contains at least one.

\begin{theorem}[Classification stability]
Suppose a profile is $(R,r)$--polarized  with respect to a distance $d$, with witnessing centers $x,y$.
\begin{enumerate}
\item[(a)] If $R>4r$, every minimax-optimal pair of centers induces the same clustering, with the cast ballots in $B_r(x)$ and those in $B_r(y)$ as the two clusters.  If $R>10r$, then any 2-biLipschitz-equivalent $d'$ gives the same optimal minimax clustering.
\item[(b)] Suppose in addition that each of $B_r(x)$, $B_r(y)$ contains at least an $\alpha$ share of the cast ballots.  
Then the corresponding statements hold for sum-optimal clustering, with $R>(4+2/\alpha)\,r$ and
 $R>(10+8/\alpha)\,r$ sufficient for stable $d$ and $d'$ clustering, respectively.
\end{enumerate}
\end{theorem}

\begin{proof}
(a) Of the pair of optimal centers, let $c_x$ be the one closer to some ballot in $B_r(x)$. Since $x,y$ themselves give a minimax cost of at most $r$, we have $d(c_x,x)\le 2r$.  Then for a ballot in $B_r(y)$, its nearest center is within $r$ while $c_x$ is at least $R-3r$ away.  The classification is determined as long as $R-3r>r$.  With respect to $d'$, the optimal cost is at most $2r$, so $d(c_x,x)\le 5r$; then a ballot in $B_r(y)$ is within $2r$ of its nearest center and at least $\frac12(R-6r)$ from $c_x$, so we need $\frac12(R-6r)>2r$ to get our guarantee.

(b) Now the cost is summed, so we don't need a center near every ballot, just close enough on average.  Since $x,y$ give a summed cost of at most $nr$ over $n$ ballots, and $B_r(x)$ contains at least $\alpha n$ of them, some ballot in $B_r(x)$ has a center $c_x$ within $r/\alpha$, so $d(c_x,x)\le r+r/\alpha$; likewise some center $c_y$ has $d(c_y,y)\le r+r/\alpha$.  A ballot in $B_r(y)$ is then within $2r+r/\alpha$ of $c_y$ and at least $R-2r-r/\alpha$ from $c_x$, so the classification is determined when $R>(4+2/\alpha)r$.  With respect to $d'$, the summed cost is at most $2nr$, so the same averaging gives $d(c_x,x)\le r+4r/\alpha$ and $d'(c_y,y)\le 2r+2r/\alpha$; a ballot in $B_r(y)$ is then within $4r+2r/\alpha$ of $c_y$ and at least $\frac12(R-2r-4r/\alpha)$ from $c_x$, so we want $R>(10+8/\alpha)r$.
\end{proof}

The bounds $d_B\le d_H<2d_B$ (Corollary~\ref{C:hihi}) let us apply this to our principal metrics.  Some balance hypothesis (here, given by $\alpha$) is needed for sum-optimal clustering: if one covering ball contains a single ballot while the other contains many ballots that are split within the ball, then it can be cheaper for the sum objective to  strand the lone ballot, so the canonical partition may not be optimal.

A statement like this can clearly be extended to say something meaningful even if only most (rather than all) points lie in two separated balls:  even with outliers, we obtain a statement requiring any optimal clusterings to substantially agree.

The coordinate embeddings enable us to use standard heuristic clustering  algorithms to find approximate centers.  Furthermore, at the scale of Scottish elections, we can certify globally optimal 1-Kemeny centers in all 1070 elections and globally optimal 2-Kemeny solutions in virtually all elections with up to 10 candidates.  

The most popular class of clustering methods is $k$-means, which seeks to minimize the sums of squares of $L^2$ distances to $k$ centers in $\R^N$.  
For instance, Lloyd's algorithm is an iterative method that is computationally lightweight.  But the standard cost-minimizing formulation relates to $L^1$ rather than $L^2$ distances, suggesting use of the PAM (Partitioning Around Medoids) algorithm, which addresses exactly that minimization problem over candidate centers within the dataset. (That is, in our setting, PAM centers must correspond to ballots that were actually cast in the election.)  We can also consider allowing centers to vary over valid ballots, whether chosen by voters or not; and the computationally simplest problem of all would be to simply take coordinatewise medians, which may not correspond to ballots at all.


\subsection{Grouping the candidates}\label{subsec:agglom}

Another perspective on the problem is to introduce notions of similarity and dissimilarity on the candidates based on the preference profile.  By \emph{slate clustering}, we mean any method that  partitions the candidates into $k$ \emph{slates}. If desired, this can be used to induce a partition of the voters into {\em blocs} by associating each ballot with the slate to which it is (in some sense) closest. 

One very reasonable approach is to regard candidates as more similar when they tend to be ranked in closer positions by votes.

\begin{definition}[Candidate dissimilarity by rank difference]\label{D:dcand}
For any candidates $A_i,A_j\in \mathcal{C}$,
set $$\dcand(A_i,A_j):= \avg_\sigma \left|\bb(\sigma)_i-\bb(\sigma)_j\right|,$$
where the average is over all ballots $\sigma$ in the profile.
\end{definition}

The notation $D_B$ for a metric on candidates is used to emphasize that, like $d_B$ as a metric on ballots, this definition is constructed through rank differences.
As usual, we can apply this with either pessimistic or averaged convention. If we use the pessimistic convention, then unmentioned candidates all have a rank of zero, so unpopular candidates will be considered highly similar.  Depending on the application, this may or may not be desired behavior.

Under the averaged convention, the point $\bb(\sigma)$ is the centroid of the set of completions $\overline{\sigma}$ that extend $\sigma$ to a complete ballot.  For most of this paper, we have been content to work with the centroid.  For this application, we may learn more by dealing with the whole point cloud.

\begin{definition}[Candidate dissimilarity with point clouds]
For a ballot $\sigma$, let $\overline{\sigma}$ be the set of complete ballots that can be reached from $\sigma$ by a sequence of extensions.

Then for any candidates $A_i,A_j\in \mathcal{C}$,
set $$\dcandbar(A_i,A_j):= \avg_\sigma \avg_{\tau\in\overline{\sigma}} 
\left| \tau(i)-\tau(j) \right|,$$
where the first average is over all ballots $\sigma$ in the profile, and the second average is over all extensions of $\sigma$ to a complete ballot. 
\end{definition}


It is clear that $\dcand$ is a metric on candidates who receive rankings, but for $\dcandbar$ this requires proof.

\begin{proposition}$\dcandbar$ is a metric on the set of candidates.
\end{proposition} 

\begin{proof} 
Clearly $\dcandbar$ is symmetric and positive definite.  To verify the triangle inequality, let $S$ denote the set of all completions of all cast ballots with the following redundancy convention.  If the ballot $\sigma$ from a given voter has length $k$, then the $(m-k)!$ distinct completions of $\sigma$ will each appear $\frac{m!}{(m-k)!}$ times in $S$. Thus each voter is represented  $m!$ times in $S$.  Notice that $\dcandbar(A,B)$ equals the average over all $\tau\in S$ of $|\tau(A)-\tau(B)|$. The triangle inequality follows.
\end{proof}

Now that we have inter-candidate distances, we can use a variety of methods to group the candidates into slates, and from there to assign corresponding blocs of voters.  We will focus on two among many possible methods.
\begin{itemize}
\item {\bf Centers.}
First cluster the candidates:  given $k$, identify the $k$ candidates as centers that minimize the summed  $\dcand$-distances to the centers.  The $i$th slate $\C_i$ includes those candidates closest to center $i$. 
For each slate $\C_i$, define the associated slate point to be $\bb\bigl((\C_i,\C\setminus\C_i)\bigr),$
the Borda image of the generalized ballot that ties all members of $\C_i$ above all remaining candidates.
Now embed the ballots into $\R^m$ via $\bb$ and partition them by associating each ballot to its $L^1$-closest slate point.

    \item {\bf Agglomeration.} Agglomerate the candidates with respect to $\dcandbar$ and average linkage.\footnote{Agglomerative, or hierarchical, clustering requires a {\em linkage} choice: single, average, or complete. These merge point clouds successively by looking at their mutual distances as measured by minimum, average, or maximum distance of their constituent points, respectively. An example is carried out below in Figure~\ref{fig:agglom}.}  To obtain $k$ slates, stop after the appropriate number of agglomeration steps.  Now associate each ballot $\sigma$ to the slate to which it awards the most Borda points per candidate.
\end{itemize}
The first method is the most in keeping with our primary clustering methods for voters.  Since the set of candidates is small, the agglomerative clustering used in the second method is also attractive, because it yields an interpretable sequence of merges starting from the candidates as singletons and ending with a pair of slates.

\FloatBarrier
\subsection{Synthetic elections}\label{sec:synthetic}

One challenge for testing clustering methods on real-world data is the lack of ground truth about whether voter behavior was polarized---not finding a good nontrivial clustering can indicate a poor clustering method or an unpolarized election.  It is therefore desirable to build a test set of synthetic elections with controlled polarization attributes.

Given $p\in(0,1)$ and a complete ballot $\sigma\in\Omega_m$, define a distribution on complete ballots by starting from $\sigma$ and taking a random walk on the complete portion of the ballot graph, with as many steps as a $p$-weighted coin comes up tails, in the style of a Mallows model. Given $n\in\N$, let $C(\sigma,n,p)$ denote a sample of $n$ independent ballots from this distribution.  A higher $p$ value corresponds to a tighter cluster (because the coin lands heads more quickly so the walks are shorter).  A simple polarized election can then be formed from two or more such clusters, forming a mixture of random-walk clusters that resembles a mixed Mallows model.  We first consider an election-valued random variable, identified with its distribution on ballots: 
$$\E=\C(\text{ABCDE}, 300, 0.3)\cup \C(\text{EDCBA} , 700, 0.3)\in\Delta(\Omega_5).$$

\noindent An instance is shown in Figure~\ref{F:Mallow_MDS}.
Here, we use a standard technique called {\em multidimensional scaling} (MDS) for the visualization:  from the full matrix of pairwise distances between points, we find a low-distortion planar embedding in order to visualize the metric.  We will employ MDS for visualization throughout the empirical sections of the paper.

\begin{figure}[bht!]\centering
\includegraphics[width=3.8in]{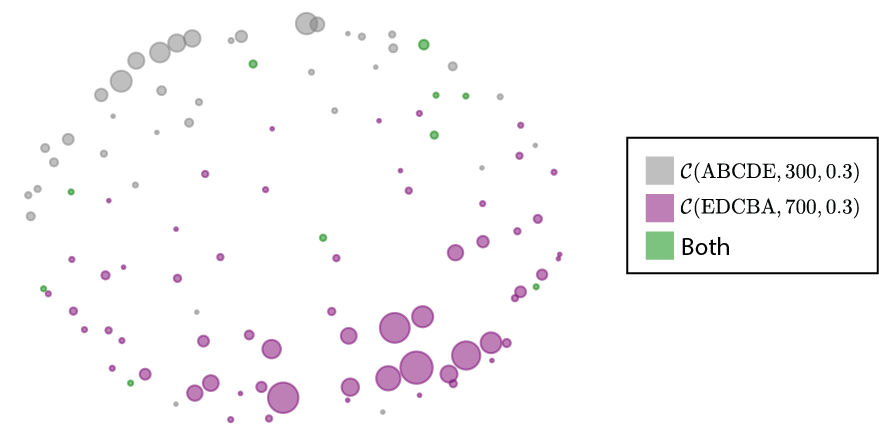}
\caption{MDS plot of the ballots in a synthetic election $\E$. 
 Ballots marked green were hit from both centers.}\label{F:Mallow_MDS}
\end{figure}

All of our modeling choices for $2$-clustering the voters pick out the ground-truth clusters in this election, even though the cluster sizes are unbalanced.  For example, over several instances of $\E$ and several runs of PAM, the identified centers are always $\{ABCDE, EDCBA\}$, with cluster sizes close to $\{300,700\}$.  That is, it finds the correct centers and identifies each ballot with its closest center rather than the site from which it originated.  Even though Lloyd's algorithm for the $2$-means problem minimizes squared $L^2$ cost rather than $L^1$, it gives clusters that differ by only about $1\%$ from the $L^1$ clusters.\footnote{Two partitions differ by $1\%$  if $1\%$ of the points must be reclassified to convert one partition into the other.}  The agglomerative method discussed in the previous section yields the reasonable slates $\{\{A,B, C\},\{D,E\}\}$ and produces voter clusters that differ by only about $4\%$ from the previous ones, with cluster sizes (in one instance of $\E$) of $\{317, 683\}$.

\begin{figure}[htb!]\centering
\includegraphics[width=5.2in]{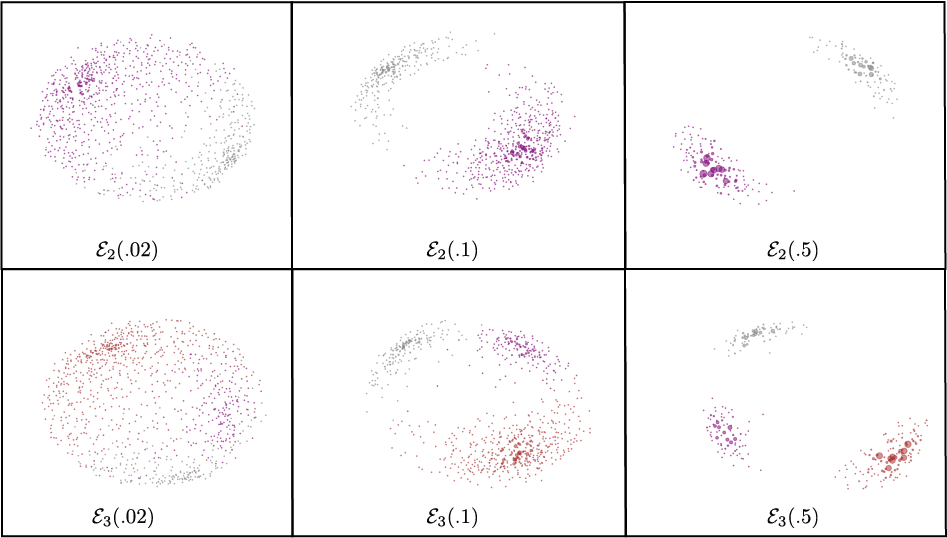}
\caption{MDS plots of the ballots in $\E_2(p)$ and $\E_3(p)$ for three values of $p$.}\label{F:E1E2E3}
\end{figure}

\begin{figure}[htb!]\centering
\includegraphics[width=5.2in]{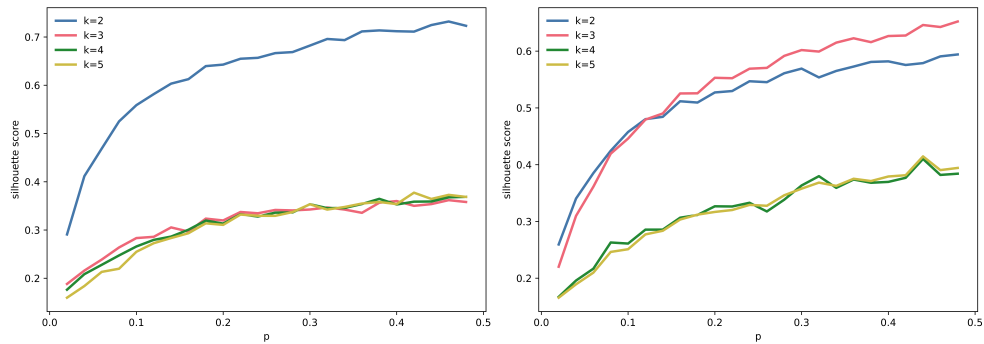}
\caption{Silhouette scores for Lloyd's $k$-means clusterings of $\E_2(p)$ (left) and $\E_3(p)$ (right) for values of $p$ from $.02$ to $.5$ by steps of $.02$.  (We show $k$-means instead of PAM clusters  for computational efficiency.)}\label{F:Sil_p}
\end{figure}

Next, consider $9$-candidate elections $\E_2(p)$ and $\E_3(p)$,  with instances shown in Figure~\ref{F:E1E2E3}.
\begin{align*}
\E_2(p)  & =  \C(\text{ABCDEFGHI}, 300, p)
 \cup\C(\text{HGEIFCBAD}, 700, p);\\              
\E_3(p)  & =  \C(\text{ABCDEFGHI}, 200, p)
                            \cup\C(\text{DFEAHBGCI}, 200, p)
                            \cup\C(\text{HIGDEFCBA}, 600, p)
\end{align*}

For \emph{all} values of $p$ between $0.02$ and $0.5$ by steps of $0.02$, the PAM 2-medoids algorithm correctly identified the two centers of $\E_2(p)$ as $\text{ABCDEFGHI}$ and $\text{HGEIFCBAD}$.  This remained true over multiple runs of PAM and multiple instances of $\E$.
For \emph{almost all} values of $p$ in this same range, the PAM 3-medoids algorithm correctly identified the three centers of $\E_3(p)$; the only discrepancies occurred with $p=0.02$ and $p=0.04$; in some trials with these values, one or two of the medoids were off by a single swap move.

There are several common methods for identifying the "best" choice for the number of clusters in a dataset.
One common method is to construct {\em silhouette scores}, which compare a point's average distance to other points within its cluster against the average distance to points in the nearest other cluster.
These are normalized so that, in principle, if the silhouette score for  the best 3-clustering is larger than the score for the best 2-clustering, it signals that the data is better described by three clusters.
Figure~\ref{F:Sil_p} shows the silhouette scores for our synthetic elections.   
For $\E_2(p)$, the silhouette method clearly picks out two clusters, even at very low values of $p$.  However,  for $\E_3(p)$, using two or three clusters looks equally apt for $p\leq 0.14$, and better than using four or five.
 
\section{Results on Scottish elections}

We begin with a detailed walk-through of implementing the clustering methods on a particular (arbitrary) contest from the Scottish local government elections:  
Edinburgh Ward 2 (Pentland Hills), 2017. From there, we will zoom out to an overview of the whole Scottish dataset.

\subsection{Learning voter blocs in Pentland Hills}

The Pentland Hills election had seven candidates competing for three seats.

\begin{center}
{\small 1=Graeme Bruce (\Con), 2 = Emma Farthing (\LD), 3 = Neil Gardiner (\SNP), \\ 4 = Ricky Henderson (\Lab), 5 = Ernesta Noreikiene (\SNP), \\ 6= Susan Webber (\Con), 7 = Evelyn Weston (\Grn).  }
\end{center}

Here, \Con is the Conservative and Unionist Party, \LD is Liberal Democrats, \SNP is the (ruling, centre-left) Scottish National Party, \Lab is Labour, and \Grn is the Green Party.  
These five  are the most important players in local government, together holding \num{1071} out of \num{1227} seats at the time of writing.  Independent candidates,  not to be confused with candidates for the Independence for Scotland Party (ISP), also receive significant support.

\begin{figure}[bht!]\centering
\includegraphics[width=6in]{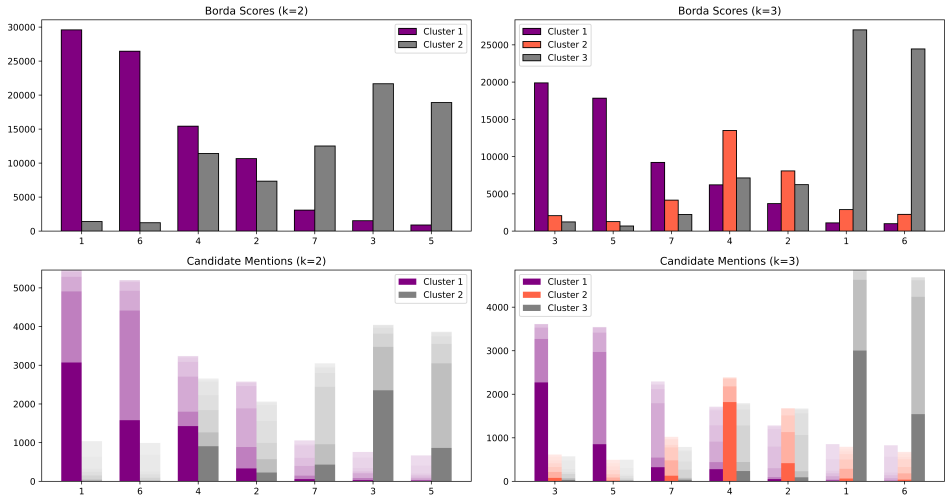}
\caption{Borda plots (top) and mentions plots (bottom) of the optimal $k$-clusterings using Borda distance with $k=2$  and $k=3$.  The total height of each bar in a "mentions" plot is the number of ballots in the cluster that ranked the candidate in any position, with first-place rankings represented by the darkest shade and so on.}\label{fig:borda-histograms}
\end{figure}

First, we describe the preference profile.
The 11,315 cast ballots in this contest ranged from bullet votes (length $1$, occurring \num{967} times) to complete ballots (length $7$, occurring 1431 times). Overall, the average ballot length was $3.2$.
There are 18 ballots cast over 100 times each, led by the ballot $(1,6)$ ranking the conservatives in alphabetical order and leaving the rest blank---that ballot was cast by 1342 voters, about 11.8\% of the electorate.  Together, the ones receiving at least \num{100} votes  account for more than half of the votes, but there is an extremely long tail, with \num{660} ballots receiving exactly one vote.
Overall,  \num{1238} different types were cast out of \num{8659}  possible valid ballots.

In Figure~\ref{fig:borda-histograms}, we show the global $d_B$ optima for minimizing cost over all ballots for $k=2$ and $k=3$.  The $k=2$ solution partitions the voters into two clusters of roughly equal size (\num{6513} in Cluster 1 and \num{4802} in Cluster 2).  The $k=3$ solution has clusters of size $3750$, $2488$ and $5077$.

\begin{figure}[htb!]\centering
\includegraphics[width=5in]{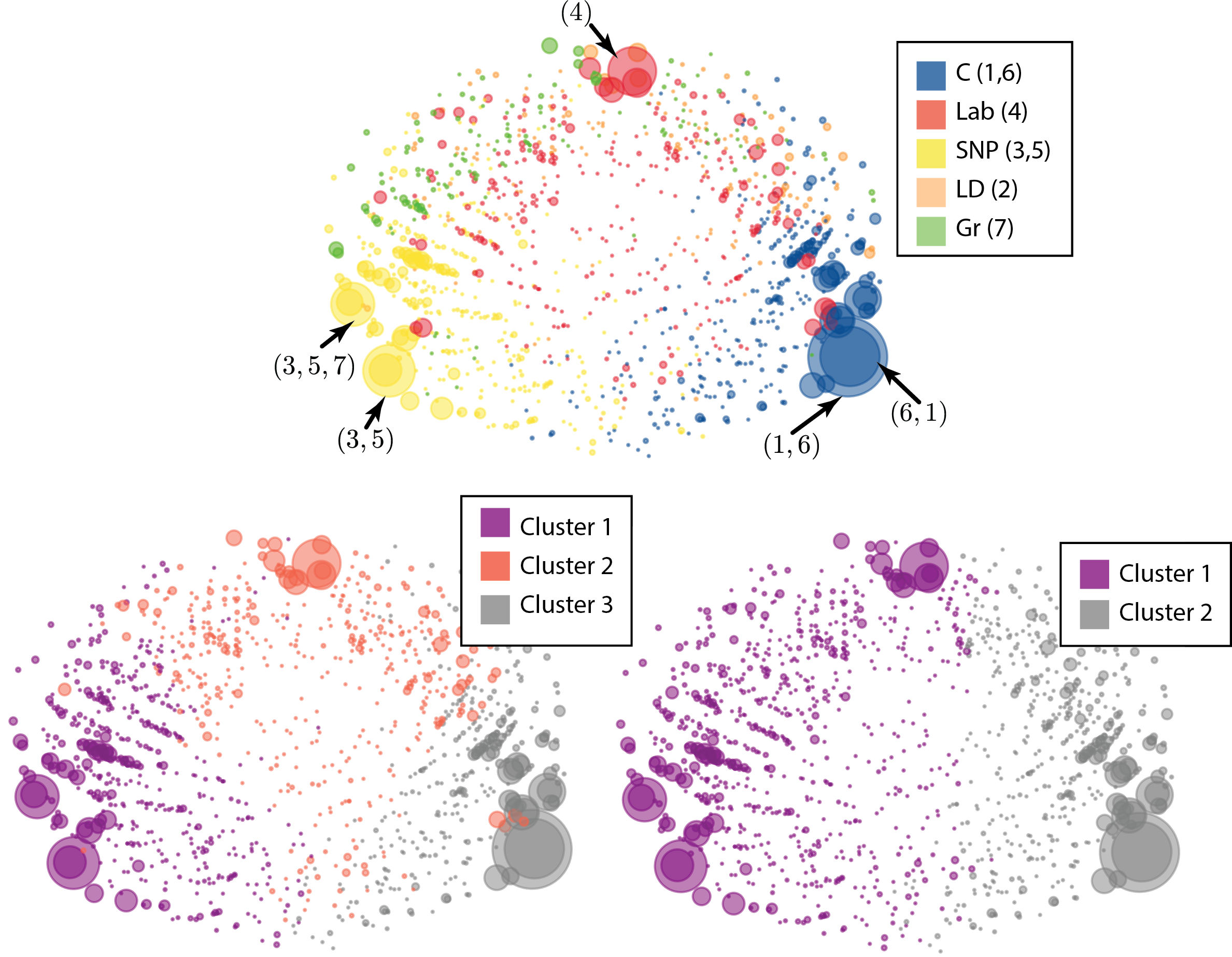}
\caption{Visualization  of the ballots in Pentland Hills with respect to $\bb$.  
All three images use the same low-distortion MDS embedding of the ballots, approximating their $d_B$ distances.
Each marker's area is proportional to the number of voters who cast that ballot.   
In the top plot, the ballots are colored by the party of the first-place vote, and on the bottom they are colored by their cluster assignment for $k=2$ and $k=3$ clusters.}\label{F:MDS}
\end{figure}

Whether we use head-to-head,  pessimistic, or averaged Borda, whether we use $L^2$ or $L^1$ methods, and whether we require cast-ballot centers, valid ballots, or allow arbitrary centers in coordinate space, we obtain very similar centers and very similar groupings of voters.    
The $k=2$ cast-ballot centers (the pair of ballots that minimize the cost) always includes $(1,6)$ as one center.  The other center always begins with $(3,5)$ and may be $(3,5,7,4)$, $(3,5,7)$, or $(3,5,4)$ depending on the method.
For $k=3$, all of the embedding choices produce the same centers: $(1, 6), (3, 5, 7), (2, 4)$.

What is the right choice for the number of clusters in  Pentland Hills?  
Silhouette scores identify $k=2$.  However, the choice $k=3$ also scores well and is compelling because, as illustrated in Figure~\ref{F:MDS}, the  $3$-clustering assignments  align closely with the party of the first-place vote,   encoding the \SNP--\Lab--\Con triangle that stands out in the party labeling (Figure~\ref{F:MDS}).

\FloatBarrier
\subsection{Learning candidate slates in Pentland Hills}

Next, we sort the candidates rather than the voters.  
As illustrated in Figure~\ref{fig:agglom}, the slate clusters are identical whether we use the centers method based on raw rank difference ($\dcand$) or the agglomerative method with ballot completion ($\dcandbar$).\footnote{Switching to complete or single linkage instead of average linkage makes only small changes.  At $k=2$ partition, the slating  remains unchanged under all linkage methods. The $k=3$  partition remains the same using complete linkage, whereas single linkage produces the slates 
$\{1,6\}, \{4\}, \{2,3,5,7\}$.}  A comparison of Figures~\ref{F:MDS} and \ref{fig:agglom} also shows that these inter-candidate differences echo the embedded structure of the full preference profile.

The candidate-centric methods agree that 
$\{1,6\}$ is a candidate cluster whether there are two clusters or three.  That squares well with the identification of $(1,6)$ as a ranked ballot appearing as central to blocs across all voter clustering methods in the previous section.  One difference is whether the Green candidate (7) is grouped with SNP candidates (3 and 5) or with the Labour/Liberal Democrat candidates (2 and 4).  As Figure~\ref{fig:agglom} shows, either choice is close to optimal in the space of candidates.


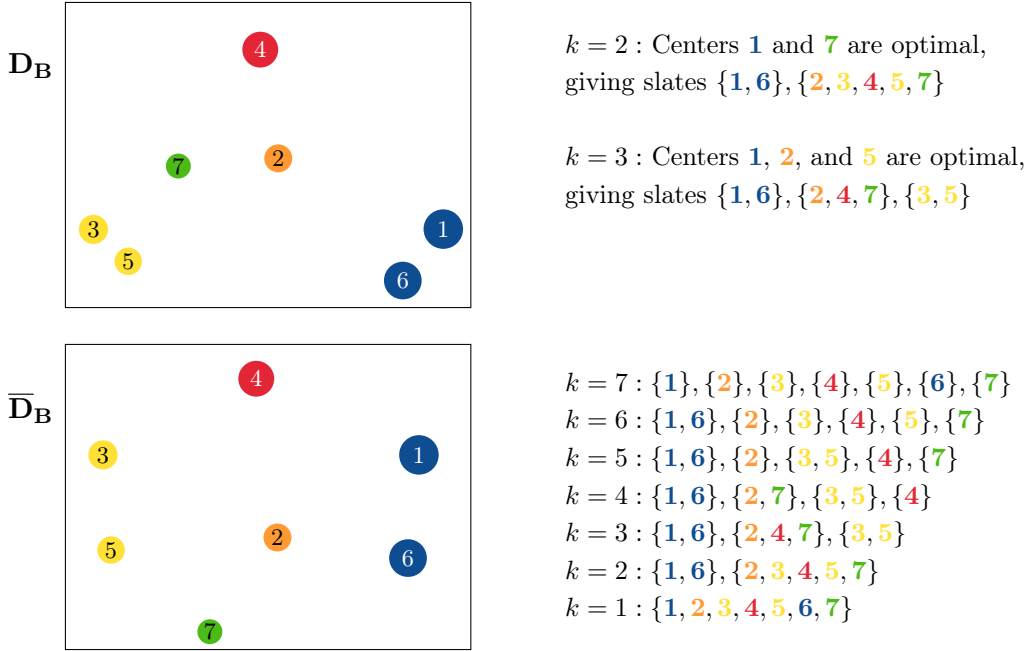
\begin{figure}[bht!]\centering
\newcommand{\one}{{\color{Concolor}\mathbf 1}}
\newcommand{\two}{{\color{LDcolor}\mathbf 2}}
\newcommand{\three}{{\color{SNPcolor}\mathbf 3}}
\newcommand{\four}{{\color{Labcolor}\mathbf 4}}
\newcommand{\five}{{\color{SNPcolor}\mathbf 5}}
\newcommand{\six}{{\color{Concolor}\mathbf 6}}
\newcommand{\seven}{{\color{Gcolor}\mathbf 7}}

\begin{tikzpicture}
\tikzset{
  p/.style={circle, inner sep=0pt, font=\small}
}

\begin{scope}[yshift=4cm,rotate=-70]
\node[p, fill=Concolor, minimum size=2.3*6.5pt]  at ( 1.266, 2.204)[text=white] {1}; 
\node[p, fill=LDcolor,  minimum size=2.3*4.6pt]  at (-0.369, 0.464) {2}; 
\node[p, fill=SNPcolor, minimum size=2.3*4.8pt]  at (-0.321,  -2.168) {3}; 
\node[p, fill=Labcolor, minimum size=2.3*5.9pt] at (-1.810, 0.734)[text=white] {4}; 
\node[p, fill=SNPcolor, minimum size=2.3*4.5pt] at ( 0.236,  -1.879) {5}; 
\node[p, fill=Concolor, minimum size=2.3*6.2pt]  at ( 1.725, 1.464)[text=white] {6}; 
\node[p, fill=Gcolor,   minimum size=2.3*4.1pt]  at (-0.727,  -0.819) {7}; 
\end{scope}

\begin{scope}[scale=.7,yshift=.8cm]
\draw (-3.6,3-.2) rectangle (4.1,8.8-.2);
\draw (-3.6,-3.5-.2) rectangle (4.1,2.3-.2);
\node at (-3.65,7.4) [left] {\large $\mathbf \dcand$};
\node at (-3.65,.9) [left] {\large $\mathbf \dcandbar$};
  \node[p, fill=Concolor, minimum size=2.3*6.5pt] at ( 3.114,  0.000) [text=white] {1};
  \node[p, fill=LDcolor,  minimum size=2.3*4.6pt] at ( 0.427,  -1.570) {2};
  \node[p, fill=SNPcolor, minimum size=2.3*4.8pt] at (-2.886,  0.000) {3};
  \node[p, fill=Labcolor, minimum size=2.3*5.9pt] at ( 0.025, 1.448)[text=white] {4};
  \node[p, fill=SNPcolor, minimum size=2.3*4.5pt] at (-2.732,  -1.808) {5};
  \node[p, fill=Concolor, minimum size=2.3*6.2pt] at ( 2.907,  -1.961)[text=white] {6};
  \node[p, fill=Gcolor,   minimum size=2.3*4.1pt] at (-0.856,  -3.361){7};
\end{scope}

\begin{scope}[shift={(4.7,-1.3)}, x=0.75cm, y=0.7cm, line cap=round]
  \draw[thick] (-0.5,0) -- (6.5,0);
  \draw[thick] (-0.5,0) -- (-0.5,3.2);
  \foreach \y in {1,2,3} {
    \draw (-0.5,\y) -- (-0.6,\y) node[left,font=\footnotesize] {\y};
  }
  \node at (0,-0.35) {$\one$};
  \node at (1,-0.35) {$\six$};
  \node at (2,-0.35) {$\four$};
  \node at (3,-0.35) {$\two$};
  \node at (4,-0.35) {$\seven$};
  \node at (5,-0.35) {$\three$};
  \node at (6,-0.35) {$\five$};
  \draw[thick,black!70] (0,0) -- (0,1.597) -- (1,1.597) -- (1,0);
  \draw[thick,black!70] (5,0) -- (5,1.700) -- (6,1.700) -- (6,0);
  \draw[thick,black!70] (3,0) -- (3,2.184) -- (4,2.184) -- (4,0);
  \draw[thick,black!70] (2,0) -- (2,2.344) -- (3.5,2.344) -- (3.5,2.184);
  \draw[thick,black!70] (2.75,2.344) -- (2.75,2.586) -- (5.5,2.586) -- (5.5,1.700);
  \draw[thick,black!70] (0.5,1.597) -- (0.5,3.031) -- (4.125,3.031) -- (4.125,2.586);
  \draw[thin,gray] (-0.5,2.809) -- (7.4,2.809) node[above right,font=\footnotesize] {$k=2$};
  \draw[thin,gray] (-0.5,2.465) -- (6.5,2.465) node[right,font=\footnotesize] {$k=3$};
\end{scope}


\node at (4,6) [right] {$k=2:$ Centers $\one$  and $\seven$ are optimal,};
\node at (4,5.5) [right] {giving slates $\{\one,\six\},\{\two,\three,\four,\five,\seven\}$};
\node at (4,4.5) [right] {$k=3:$ Centers $\one$, $\two$,  and $\five$ are optimal,};
\node at (4,4) [right] {giving slates $\{\one,\six\},\{\two,\four,\seven\}, \{\three,\five\}$};
\end{tikzpicture}

\caption{MDS plots of the Pentland Hills candidates by $\dcand$ and $\dcandbar$, colored by party and sized by the number of mentions received. The optimal centers for $\dcand$ slate clustering are shown, as well as a dendrogram showing the steps of agglomerative $\dcandbar$ slate clustering.  The $k=2$ and $k=3$ slates match exactly across the two methods.}\label{fig:agglom}
\end{figure}

\subsection{Overview across dataset}
\label{subsec:scottish-overview}

We now zoom out to the full dataset of Scottish local government elections,
which are conducted by single transferable vote (STV). The cast vote records were
originally collected and cleaned by David McCune from Scottish council
sources; see \citet{McCune_GrahamSquire_2024} for a description. We use the version
hosted in the MGGG \texttt{scot-elex} repository \citep{mggg_scot_elex}, which
contains 1070 ward elections from the regularly scheduled council elections
of 2007, 2012, 2017, and 2022. (Thirty by-elections were set aside.) Elections have between 3 and 14 candidates,
with 7 candidates the most common, and almost all fill 3 or 4 seats.

To approach clustering on the Scottish dataset, we first compare $2$-clustering methods.  We can efficiently sweep the dataset with the heuristic methods described above. %
Then, with integer programming methods described in forthcoming work  with David Shmoys and Peter Rock, we can confirm that the heuristic optima are often in fact globally optimal.  In over $90\%$ of Scottish elections ($981$ matches with Borda embeddings and $1010$ with head-to-head), the PAM heuristic finds precisely optimal centers.  
Appendix~\ref{app:compare} offers a fuller discussion of comparison for the approximate methods.

More fundamentally, we can study the degree of agreement of exact optima across methods.  To evaluate this, we can start by regarding the center of the whole election as a summary of the preference profile---the minimum-cost cast ballot.
For $k=1$, 74.6\% of the elections have matching head-to-head and Borda centers, meaning that the Kemeny rule would give the same output ranking as the Borda median rule, even with partial ballots available.  This is notable because Kemeny is strongly Condorcet consistent---if any candidate is preferred head-to-head to all others, they rank first; if any candidate is preferred to all but the first, they rank second, and so on.  Rules based on positional scores like Borda count have no such guarantee and are a priori very different. 

When we pass to optimal pairs under cost-minimization, the exact agreement of centers is significantly more demanding.  But for $k=2$, the degree of agreement is almost as high:  73.5\% of the elections have perfectly matching $d_H$ and $d_B$ centers (i.e., both centers agree).

Though our methods for grouping candidates are party-agnostic, we can use party labels as a partial ground-truth validation.  
Both methods of slate clustering perform well at keeping political parties together.  Among the $1070$ Scottish elections, $855$ of them include more than one candidate from at least one of the five major parties (\SNP, \Lab, \Con, \LD, \Grn).  It is extremely rare (occurring eight times for the centers method and just three for the agglomerative method) for the $k=2$ slates to split any pairs of candidates from the same party.  But beyond recovering parties, we can discover inter-party structure: for instance, the two most consistently close parties across this dataset are the Scottish National Party and the Greens (rather than pairing the ruling \SNP party with Labour or Liberal Democrats, as might have been expected).  

\section{Conclusion}

Although aggregating rankings is an extremely popular topic in the computing literature, surprisingly little of that work is adapted to the real-world setting of ranked choice elections for political office, where an extension to partial rankings is a practical necessity.  First and foremost, this paper develops metric theory that situates ballots, candidates, and slates (or sets of candidates) all in common spaces that admit natural extensions of the Kendall tau and Spearman footrule metrics.  This enables the use of coordinate tools, on one hand, and graph/network-based tools, on the other.

Building on this foundation, we give detailed methods for grouping voters into like-minded blocs and for grouping candidates according to voter preference.  Bloc and slate identification then unlocks the measurement of key descriptive statistics:  to measure the {\em polarization} of an election, we look at the diffusion of blocs about their centers and the distance between those centers.  As a direct consequence of inequalities connecting the principal metrics under study, sufficiently polarized elections must give consistent partitions of voters by either ranking metric.  Given blocs and slates, it is now straightforward to assess the {\em proportionality} of a voting rule by whether the slates receive representation in shares aligned with the sizes of voter blocs.  

The techniques developed here broadly enable visualization and hands-on data analysis, while giving a strong foundation for unsupervised techniques to learn structural information about observed elections.  

\vspace{40pt}

\section*{Acknowledgements}
The authors are pleased to thank David Shmoys, Peter Rock, Chris Donnay, Edith Elkind, and Grant Rinehimer for overlapping collaborations and valuable feedback.  We also thank the referees for comments that greatly improved the paper.  Browser-based AI tools were used to review the manuscript after completion. The authors assume responsibility for all content.

\bibliographystyle{plainnat}
\bibliography{bibliography}

\clearpage
\appendix

\section{Comparing heuristic methods}\label{app:compare}

Table~\ref{T:methods} indicates a high degree of concordance among outputs over our database of 1070 Scottish council elections between 2007 and 2022. In particular, the choice of embedding ($\bb$ or $\hh$) is generally of fairly mild consequence.  Although we did not include it in the table, $\bb_{\text{avg}}$ performed very similarly to $\bb$.  Figure~\ref{F:method-MDS-bloc-size} (left) uses MDS to visualize the distances reported in the table. Recall that Lloyd's algorithm returns coordinatewise means, while PAM returns medoids. In the Borda embedding, the Lloyd centroids record average Borda scores; in the head-to-head embedding, they record average pairwise-comparison scores.

\begin{table}[bht!] \centering 
\begin{tabular}{cccccc}
\toprule & Lloyd $\bb$ & Lloyd $\hh$ & PAM $\bb$ & PAM $\hh$  \\
\midrule Lloyd $\bb$ &0.00 &0.02 &0.07 &0.07  \\
Lloyd $\hh$  & -& 0.00& 0.07&0.07  \\
PAM $\bb$ &- &-&0.00 &0.04  \\
PAM $\hh$ &- &- &- & 0.00 \\
\bottomrule
\end{tabular}
\caption{The average over 1070 elections of the differences between the $2$-clustering methods.  Each cell represents the portion of ballots in the election for which the two partitions disagree on average.  Diagonal cells represent the average disagreement between two runs of the same method. 
 All cells are rounded to $2$ decimals.}\label{T:methods}
\end{table}

Next, we compare the results of slate clustering by the two methods emphasized above.  To compare the sizes of slates in an apples-to-apples fashion, we restrict to the 282 Scottish elections that have seven candidates (the most common number of candidates).  For the separation of candidates into two slates, we report the size of the smaller slate. 
\begin{itemize}
\item Via centers: 35 slates of size 1, 164 slates of size 2, 83 slates of size 3
\item Via agglomeration: 23 slates of size 1, 185 slates of size 2, 74 slates of size 3
\end{itemize}

\begin{figure}[bht!]\centering
\begin{tikzpicture}
\node at (0,0) {\includegraphics[height=1.8in]{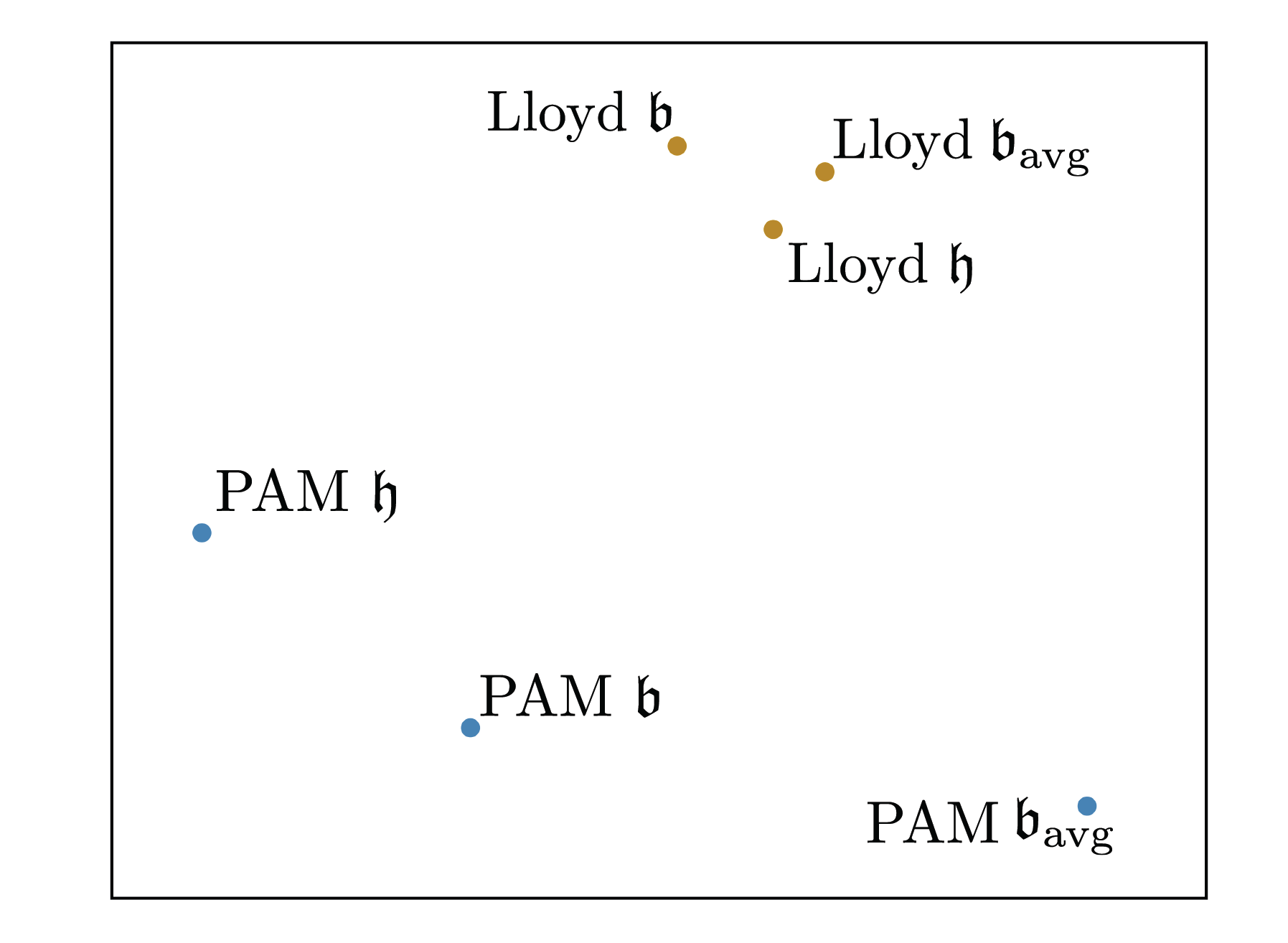}}; 
\node at (6.5,-.3) {\includegraphics[height=1.8in]{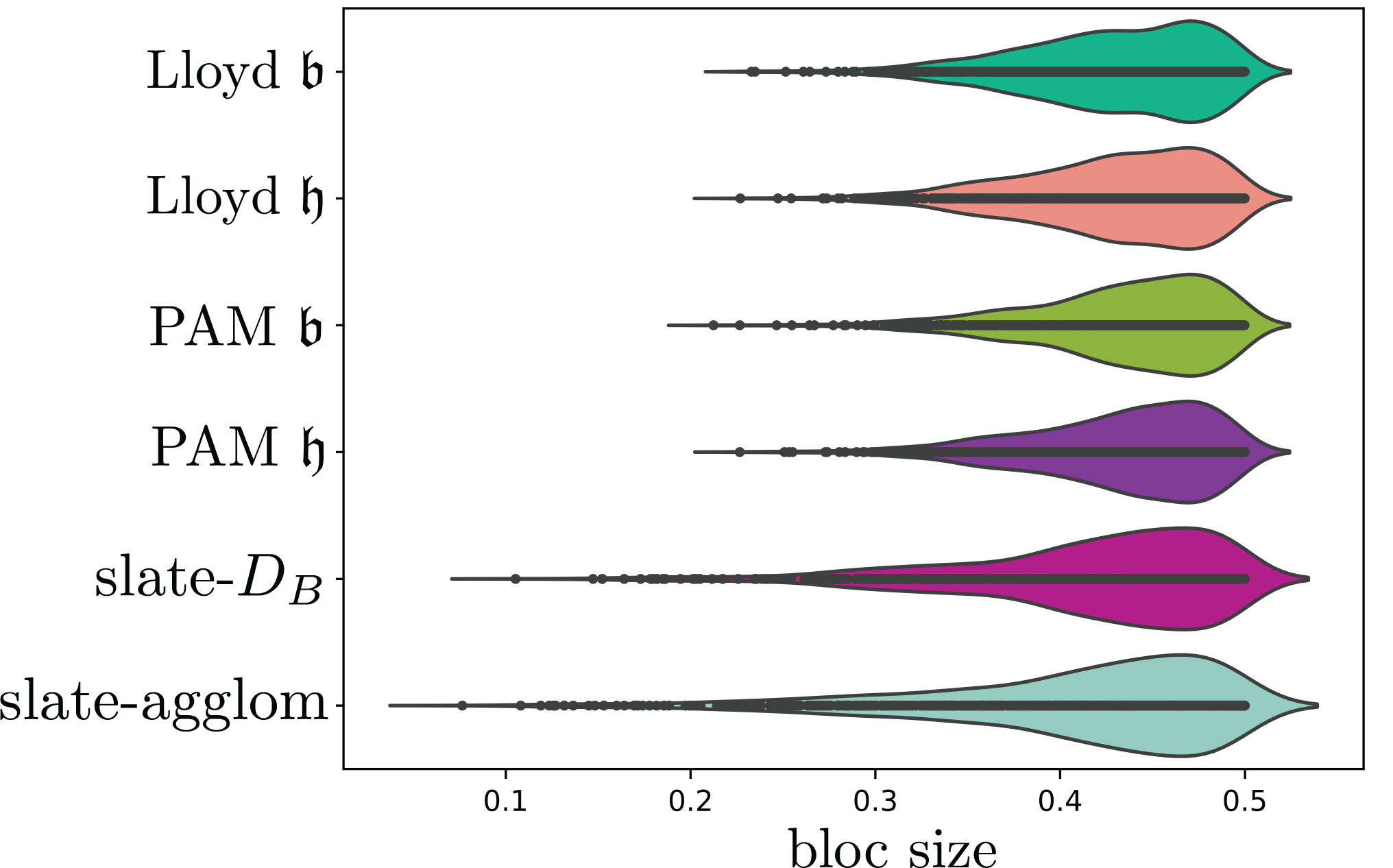}}; 
\end{tikzpicture}
\caption{LEFT: MDS plot of distances between clustering methods from Table~\ref{T:methods}. 
 RIGHT: A visualization of how balanced the cluster sizes are in the 2-bloc case.  Smaller bloc shown.}\label{F:method-MDS-bloc-size}
\end{figure}

Figure~\ref{F:method-MDS-bloc-size} (right) shows the distribution of bloc sizes for each $2$-clustering method.  None of the clustering methods specifically targets balanced bloc sizes, but the results in Scotland tend toward balanced blocs for all clustering methods.

\section{Visualizing polarization}\label{app:slate-viz}

The agglomerative slates for $k=2$ and $3$ in the Pentland Hills election are pictured in Figure~\ref{F:heatmap} in a way that illustrates the election's level of polarization with respect to these slates.  The idea of this visualization is that the $k$ vertices of the simplex
$\Delta^{k-1} = \left\{(x_1,...,x_k)\mid \sum x_i=1, \text{all }x_i\geq 0\right\}$
represent the $k$ slates.  Each ballot $\sigma$ cast in the election has a position $f(\sigma)\in\Delta^{k-1}$ determined so that a ballot is positioned closer to a slate that it more strongly supports, and the plots illustrate the density of the cast ballot images over the simplex.  More precisely, we used the function
\begin{equation}\label{E:heat}f(\sigma) = \frac{(v_1,...,v_k)}{v_1+\cdots+ v_k},\end{equation}
where $v_i$ is the average number of (pessimistic) Borda points that $\sigma$ awards to the candidates in the $i^{\text{th}}$ slate.  

This definition of $f$ is consistent with our agglomerative slate-clustering  method of assigning ballots to slates; that is, the agglomerative method was defined to assign each ballot $\sigma$ to the slate corresponding to the vertex of $\Delta^{k-1}$ closest to $f(\sigma)$.


\begin{figure}[bht!]\centering
\includegraphics[width=5.5in]{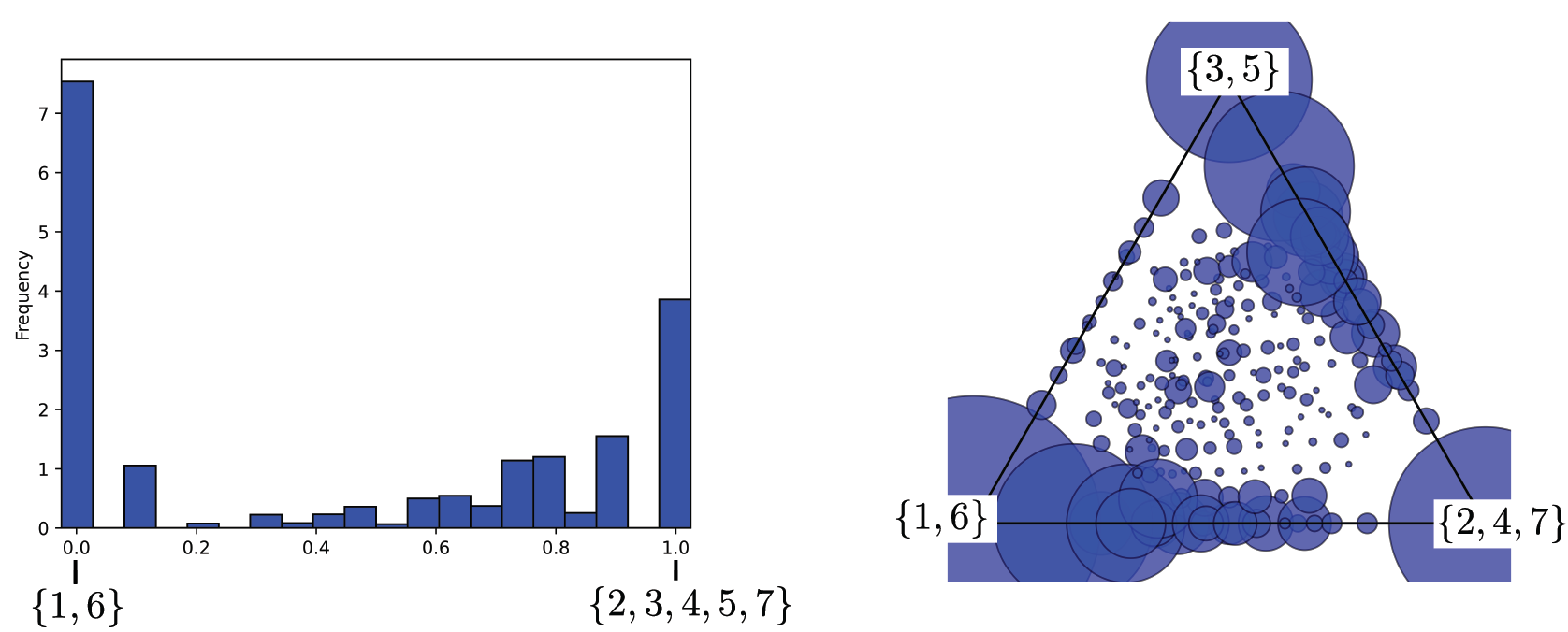}
\caption{Plots showing the density of ballots cast in the Pentland Hills election with respect to the $k=2$  and $k=3$  slates obtained from either clustering method.  On the  right, the triangle represents $\Delta^{2}$ and the area of each circle is proportional to the number of ballots $\sigma$ mapped to $f(\sigma)$.}\label{F:heatmap}
\end{figure}

For both $k=2$ and $k=3$, the set of slates in Figure~\ref{F:heatmap} arose from agglomerative clustering of the candidates, but it happens to also be the partition of the candidates that minimize the sum over the ballots $\sigma$ of the Euclidean distance in $\Delta^{k-1}$ from $f(\sigma)$ to the nearest vertex (slate).  In other words,  Figure~\ref{F:heatmap} helps to visualize the profile's level of polarization with the choice of slates maximizing that polarization.  This observation suggests another reasonable slate clustering method: directly select the slates that solve this optimization problem.  Compared to the agglomerative method, this alternative version is computationally more demanding but yields an interpretable visualization.  The two methods agree in the Pentland Hills election (because they turn out to yield the same slates, and they use the same method of associating ballots to slates).

\end{document}